\RequirePackage{fix-cm}
\documentclass[smallextended]{svjour3}
\smartqed  
\usepackage{graphicx}
\usepackage{stmaryrd,amsmath,mathrsfs}
\usepackage{tabularx, subfigure, epsfig, graphicx, cases}
\usepackage{subcaption}
\usepackage{color, CJK, anysize, enumerate}
\usepackage[colorlinks,hyperindex,CJKbookmarks,linkcolor=blue,anchorcolor=blue,citecolor=cyan]{hyperref}
\usepackage{booktabs, multirow}
\usepackage{float}
\usepackage{cite}
\usepackage{stmaryrd}
\usepackage{epstopdf}
\SetSymbolFont{stmry}{bold}{U}{stmry}{m}{n}
\usepackage{mathrsfs} %
\usepackage{amsmath}  %
\usepackage{amssymb}  %

\numberwithin{equation}{section}

\journalname{Journal of Scientific Computing}
\begin{document}
	\title{Second-Order Area/Volume-Preserving PFEMs for Surface Diffusion via Simpson--Boole Geometric Identities}
	\author{ Zhiqing Pan  \and 
		Jiwei Jia \and Lian Zhang
	}
	\institute{
		Zhiqing Pan \at
		Department of Computational Mathematics, School of Mathematics, Jilin University, Changchun, China\\
		\email{panzq24@mails.jlu.edu.cn}
		\and
		Jiwei Jia* (Corresponding author) \at
		Department of Computational Mathematics, School of Mathematics, Jilin University, Changchun, China  \\
		AI for Science and Engineering Center, Shenzhen Loop Area Institute, Shenzhen, China \\
		\email{jiajiwei@jlu.edu.cn}
		\and
		Lian Zhang \at 
		Shenzhen Research Institute of Big Data, Shenzhen, China \\ 
		AI for Science and Engineering Center, Shenzhen Loop Area Institute, Shenzhen, China  \\
		\email{zhanglian@sribd.cn}
	}
	\date{Received: date / Accepted: date}
	\maketitle

\begin{abstract}
	We propose second-order-in-time parametric finite element methods for surface diffusion of closed curves in two dimensions and closed surfaces in three dimensions. The construction is based on exact geometric variation identities along a quadratic temporal interpolation path. The induced area variation in 2D is evaluated exactly by Simpson's rule, while the induced volume variation in 3D is evaluated exactly by Boole's rule.
	
	The resulting fully discrete schemes preserve the enclosed area or volume exactly, without introducing an auxiliary Lagrange multiplier for the geometric constraint. They can be assembled on BGN-predicted auxiliary geometries and are therefore compatible with existing second-order BGN-type implementations. Numerical experiments demonstrate the expected second-order behavior, area/volume conservation, and good mesh quality for both curve and surface evolutions.
	
\end{abstract}
	
	\section{Introduction}
	
	Surface diffusion is a fourth-order geometric evolution in which the normal velocity is given by the Laplace--Beltrami operator of curvature. It arises in a variety of interfacial relaxation phenomena, including thermal grooving, crystal surface relaxation, and solid-state dewetting \cite{mullins1957thermal,mullins1959flattening,cahn1994surface}. For closed curves in two dimensions and closed surfaces in three dimensions, the flow preserves the enclosed area/volume and dissipates the perimeter/surface area. These geometric properties are fundamental both analytically and numerically, especially in long-time simulations.
	
	A variety of numerical approaches have been developed for surface diffusion and related geometric flows, including graph-based finite element methods, level-set formulations, diffuse-interface/phase-field models, and unfitted or trace finite element discretizations; see, for example, \cite{deckelnick2005fully,smereka2003semi,elliott1996cahn,olshanskii2018trace,xu2009ldg}. These approaches provide useful alternatives for handling implicit interfaces, topology changes, or nonparametric representations. In the present work, however, we focus on parametric finite element methods for closed interfaces, pioneered by Dziuk's work \cite{dziuk1994convergence,bansch2004surface,dziuk1990algorithm,dziuk2002evolution}.
	
For parametric descriptions of evolving curves and surfaces, the parametric finite element method (PFEM) developed by Barrett, Garcke and N\"urnberg (BGN) has become a particularly influential framework \cite{barrett2007parametric,barrett2008parametric1,barrett2020parametric}. A key strength of BGN-type methods is that the weak formulation naturally induces tangential motion, which typically leads to good mesh quality and, in the planar case, asymptotic equidistribution of mesh points. For surface diffusion, the classical BGN discretization is first-order in time and has proved robust in practice.

Motivated by the geometric structure of the flow, considerable effort has been devoted to designing PFEMs that preserve enclosed area/volume and respect the dissipative nature of the evolution at the discrete level. In particular, Bao and Zhao proposed a first-order structure-preserving PFEM for surface diffusion that conserves area/volume through an exact geometric identity associated with a linear interpolation between two consecutive time levels \cite{bao2021structure}. This geometric viewpoint has subsequently been extended to other geometric flows, including anisotropic surface diffusion and two-phase Navier--Stokes flow \cite{Bao2023Symmetrized,bao2023symmetrized1,bao2025unified,garcke2023structure}. Another related line of work imposes geometric constraints by Lagrange multipliers at the fully discrete level; see, for example, \cite{garcke2025structure}.

To improve temporal accuracy, several second-order PFEMs based on the BGN framework have been developed in recent years. These include Crank--Nicolson--leapfrog formulations, stable BDF2 discretizations, and predictor--corrector strategies for surface diffusion and related flows \cite{jiang2024second,jiang2024stable,jiang2025predictor}. More recently, Duan and Yang \cite{duan2026second} proposed a second-order volume-preserving PFEM for surface diffusion. Their method attains second-order accuracy and volume preservation through a Crank--Nicolson time discretization, which also draws on the linear-interpolation idea of Bao and Zhao \cite{bao2021structure}. They also incorporate a harmonic-map tangential redistribution for mesh quality and a Lagrange-multiplier treatment for energy stability.

In this work, we extend the geometric conservation idea of Bao and Zhao to the second-order setting in a different way. We derive the conservative term directly from an exact geometric variation identity generated by a quadratic temporal interpolation through three consecutive time levels. Evaluating the induced area/volume variation by Newton--Cotes quadrature yields a Simpson-type formula in 2D and a Boole-type formula in 3D. In this way, the conservative term is built intrinsically from the interpolation path itself.

This construction leads to second-order PFEMs for surface diffusion that preserve the enclosed area/volume without auxiliary Lagrange multipliers.  In particular, the resulting schemes can be viewed as a linear combination of the BGN/CN-leapfrog, BGN/BDF2 and BGN/PC schemes \cite{jiang2024second,jiang2024stable,jiang2025predictor} with a mild modification of the normal vectors.
	
	The main contributions of this paper are summarized as follows:
	\begin{itemize}
		\item We extend the geometric conservation idea of the first-order Bao--Zhao structure-preserving PFEM to the second-order setting by using quadratic temporal interpolation.
		\item We derive exact second-order discrete variation formulas for enclosed area in 2D and enclosed volume in 3D, obtained from Simpson and Boole quadrature, respectively.
		\item We show that the resulting conservative construction is compatible with predictor-based BGN-type implementations and retains the good mesh behavior observed in practice.
	\end{itemize}
	
	The rest of the paper is organized as follows. In Section~\ref{sec:2D}, we present the second-order area-preserving PFEM for closed planar curves. In Section~\ref{sec:3D}, we extend the construction to a second-order volume-preserving PFEM for closed surfaces in three dimensions. Numerical experiments are reported in Section~\ref{sec:numerics}, and concluding remarks are given in the final section.
	
	\section{A second-order area-preserving PFEM for closed curves in 2D}
	\label{sec:2D}
	
	In this section, we construct a second-order-in-time area-preserving parametric finite element method for surface diffusion of closed curves. The notation is chosen to be close to the first-order structure-preserving PFEM.
	
	Let $\Gamma(t)$ be a family of closed curves in $\mathbb{R}^2$, parametrized by
	\[
	\mathbf{X}(\rho,t):\mathbb{I}\times[0,T]\to\mathbb{R}^2,\qquad 
	\mathbb{I}=\mathbb{R}/\mathbb{Z}.
	\]
	Let $s$ be the arclength parameter, $\mathbf{n}=-(\partial_s\mathbf{X})^\perp$ the outward unit normal vector, and $\kappa$ the curvature. The surface diffusion flow is written as
	\begin{equation}
		\mathbf{n}\cdot\partial_t\mathbf{X}=\partial_{ss}\kappa,
		\qquad
		\kappa\mathbf{n}=-\partial_{ss}\mathbf{X}.
		\label{eq:strong-2d}
	\end{equation}
	It is well known that the enclosed area $A(t)$ and the perimeter $L(t)$ satisfy
	\begin{equation}
		\frac{d}{dt}A(t)=0,\qquad
		\frac{d}{dt}L(t)=-\int_{\Gamma(t)}|\partial_s\kappa|^2\,ds\le 0.
	\end{equation}
	
	The weak formulation reads: find $\mathbf{X}(\cdot,t)\in[H^1(\mathbb{I})]^2$ and $\kappa(\cdot,t)\in H^1(\mathbb{I})$ such that
	\begin{subequations}
		\begin{align}
			\left(\mathbf{n}\cdot\partial_t\mathbf{X},\varphi\right)_{\Gamma(t)}
			+\left(\partial_s\kappa,\partial_s\varphi\right)_{\Gamma(t)}
			&=0,
			&&\forall \varphi\in H^1(\mathbb{I}),
			\label{eq:weak-2d-a}
			\\
			\left(\kappa,\mathbf{n}\cdot\boldsymbol{\omega}\right)_{\Gamma(t)}
			-\left(\partial_s\mathbf{X},\partial_s\boldsymbol{\omega}\right)_{\Gamma(t)}
			&=0,
			&&\forall \boldsymbol{\omega}\in[H^1(\mathbb{I})]^2.
			\label{eq:weak-2d-b}
		\end{align}
	\end{subequations}
	
	\subsection{Spatial discretization}
	
	Let $0=\rho_0<\rho_1<\cdots<\rho_N=1$ be a uniform partition of $\mathbb{I}$ with $h=1/N$. Define
	\[
	V^h:=\left\{u\in C(\mathbb{I}):\ u|_{[\rho_{j-1},\rho_j]}\in\mathbb{P}_1,\ 
	u(0)=u(1)\right\}.
	\]
	Let $\Gamma^m=\mathbf{X}^m(\mathbb{I})$ be a polygonal approximation of $\Gamma(t_m)$ with $t_m=m\tau$ and $\mathbf{X}^m\in[V^h]^2$. For piecewise continuous functions $u,v$, define the mass-lumped inner product
	\begin{equation}
		(u,v)^h_{\Gamma^m}
		:=
		\frac12\sum_{j=1}^N |\mathbf{h}_j^m|
		\left[(u\cdot v)(\rho_j^-)+(u\cdot v)(\rho_{j-1}^+)\right],
		\qquad
		\mathbf{h}_j^m:=\mathbf{X}^m(\rho_j)-\mathbf{X}^m(\rho_{j-1}).
	\end{equation}
	
	\subsection{Quadratic interpolation and the exact area identity}
	
	The central idea is to connect $\Gamma^{m-1}$, $\Gamma^m$ and $\Gamma^{m+1}$ by a quadratic polynomial in a local time variable $\alpha\in[0,1]$:
	\begin{equation}
		\begin{aligned}
			\mathbf{X}^h(\alpha)
			&=
			(2\alpha^2-3\alpha+1)\mathbf{X}^{m-1}
			+(-4\alpha^2+4\alpha)\mathbf{X}^{m}
			+(2\alpha^2-\alpha)\mathbf{X}^{m+1}.
		\end{aligned}
		\label{eq:quad-X-2d}
	\end{equation}
	Then
	\[
	\mathbf{X}^h(0)=\mathbf{X}^{m-1},\qquad
	\mathbf{X}^h(1/2)=\mathbf{X}^{m},\qquad
	\mathbf{X}^h(1)=\mathbf{X}^{m+1}.
	\]
	Since the interval $[t_{m-1},t_{m+1}]$ has length $2\tau$, we have
	\[
	t=t_{m-1}+2\tau\alpha,
	\qquad
	\partial_t=\frac{1}{2\tau}\partial_\alpha.
	\]
	Evaluating $\frac{1}{2\tau}\partial_\alpha\mathbf{X}^h(\alpha)$ at $\alpha=0,1/2,1$ gives
	\begin{align}
		\partial_t \mathbf{X}^{m-1}
		&:=
		\frac{-\mathbf{X}^{m+1}+4\mathbf{X}^{m}-3\mathbf{X}^{m-1}}{2\tau},
		\label{eq:D-mminus}
		\\
		\partial_t \mathbf{X}^{m}
		&:=
		\frac{\mathbf{X}^{m+1}-\mathbf{X}^{m-1}}{2\tau},
		\label{eq:D-m}
		\\
		\partial_t \mathbf{X}^{m+1}
		&:=
		\frac{3\mathbf{X}^{m+1}-4\mathbf{X}^{m}+\mathbf{X}^{m-1}}{2\tau}.
		\label{eq:D-mplus}
	\end{align}
	These formulas are the quadratic special cases of the finite-difference coefficients induced by polynomial interpolation on equidistant temporal nodes; see Appendix~\ref{app:fd-coeff} for a general statement. Here $\partial_t \mathbf{X}^{m+1}$ is exactly the BDF2 approximation used in BGN/BDF schemes \cite{jiang2024stable}, $\partial_t \mathbf{X}^{m}$ is the Crank-Nicolson approximation used in BGN/CN-leapfrog schemes, and $\partial_t \mathbf{X}^{m-1}$ is the corresponding second-order backward endpoint approximation.
	
	Let $A(\alpha)$ be the area enclosed by $\Gamma^h(\alpha)=\mathbf{X}^h(\alpha)(\mathbb{I})$. Then
	\begin{equation}
		\begin{aligned}
			\frac{d}{d\alpha}A(\alpha)
			&=
			\int_{\Gamma^h(\alpha)}
			\partial_\alpha\mathbf{X}^h(\alpha)\cdot \mathbf{n}^h(\alpha)\,ds
			\\
			&=
			\int_{\mathbb{I}}
			\partial_\alpha\mathbf{X}^h(\alpha)\cdot
			\left[-\partial_\rho\mathbf{X}^h(\alpha)\right]^\perp\,d\rho.
		\end{aligned}
		\label{eq:dA-dalpha}
	\end{equation}
	Since $\partial_\alpha\mathbf{X}^h(\alpha)$ is linear in $\alpha$ and $\partial_\rho\mathbf{X}^h(\alpha)$ is quadratic in $\alpha$, the integrand in \eqref{eq:dA-dalpha} is a polynomial of degree at most three in $\alpha$. Therefore Simpson's rule is exact and gives
	\begin{equation}
		\begin{aligned}
			A^{m+1}-A^{m-1}
			& = \frac{1}{6}\int_{\mathbb{I}}\left[-\mathbf{X}^{m+1}+4\mathbf{X}^m-3\mathbf{X}^{m-1}\right] \cdot\left[-\partial_\rho \mathbf{X}^{m-1}\right]^{\perp} \mathrm{d} \rho  \\ 
			& + \frac{4}{6}\int_{\mathbb{I}}\left[\mathbf{X}^{m+1}-\mathbf{X}^{m-1}\right] \cdot\left[-\partial_\rho \mathbf{X}^m\right]^{\perp} \mathrm{d} \rho \\
			& + \frac{1}{6}\int_{\mathbb{I}}\left[3\mathbf{X}^{m+1}-4\mathbf{X}^m+\mathbf{X}^{m-1}\right] \cdot\left[-\partial_\rho \mathbf{X}^{m+1}\right]^{\perp} \mathrm{d} \rho .
		\end{aligned}
		\label{eq:area-identity-2d}
	\end{equation}
	
	\begin{remark}
		The identity \eqref{eq:area-identity-2d} is the geometric reason for exact area preservation. It is not a usual BDF2 identity, but an exact area variation formula along the quadratic interpolation path \eqref{eq:quad-X-2d}. The three velocity approximations naturally correspond to the three time levels $m-1$, $m$ and $m+1$. 
	\end{remark}
	
	\subsection{Prediction and the second-order area-preserving scheme}
	
	In the practical scheme, the new-time curve used in the differential operators is replaced by a predicted curve $\widetilde{\Gamma}^{m+1}:=\widetilde{\mathbf{X}}^{m+1}(\mathbb{I})$. Following the BGN/BDF strategy, $\widetilde{\mathbf{X}}^{m+1}$ is obtained by one step of the classical first-order BGN scheme from $\Gamma^m$ with time step $\tau$: find $\widetilde{\mathbf{X}}^{m+1}\in[V^h]^2$ and $\widetilde{\kappa}^{m+1}\in V^h$ such that
	\begin{subequations}
		\label{eq:BGN1-predictor-2d}
		\begin{align}
			\left(
			\frac{\widetilde{\mathbf{X}}^{m+1}-\mathbf{X}^{m}}{\tau},
			\varphi^h\mathbf{n}^{m}
			\right)^h_{\Gamma^{m}}
			+
			\left(
			\partial_s\widetilde{\kappa}^{m+1},
			\partial_s\varphi^h
			\right)_{\Gamma^{m}}
			&=0,
			\qquad \forall \varphi^h\in V^h,
			\\
			\left(
			\widetilde{\kappa}^{m+1},
			\mathbf{n}^{m}\cdot\boldsymbol{\omega}^h
			\right)^h_{\Gamma^{m}}
			-
			\left(
			\partial_s\widetilde{\mathbf{X}}^{m+1},
			\partial_s\boldsymbol{\omega}^h
			\right)_{\Gamma^{m}}
			&=0,
			\qquad \forall \boldsymbol{\omega}^h\in[V^h]^2.
		\end{align}
	\end{subequations}
The predicted curve is used as an auxiliary new-time geometry for assembling the mass-lumped and stiffness terms and for evaluating the metric factor in the normal vector below. The exact area conservation is not affected, because the normal in the normal velocity term is designed by
	\begin{equation}
		\mathbf{n}^{m+1}
		:=
		-\frac{(\partial_\rho\mathbf{X}^{m+1})^\perp}
		{|\partial_\rho\widetilde{\mathbf{X}}^{m+1}|}.
		\label{eq:sp-normal-2d}
	\end{equation}
	Thus $|\partial_\rho\widetilde{\mathbf{X}}^{m+1}|\mathbf{n}^{m+1}=-(\partial_\rho\mathbf{X}^{m+1})^\perp$, and the metric factor in the mass-lumped inner product on $\widetilde{\Gamma}^{m+1}$ reproduces the exact oriented area density in \eqref{eq:area-identity-2d}.
	
	\begin{remark}[Relation with BGN-based second-order methods]
		The use of a BGN1-predicted geometry follows the BGN/BDF strategy and avoids the mesh deterioration that may arise from direct extrapolation. Although the middle term below contains the centered CN-leapfrog approximation, no additional mesh redistribution step is introduced; the good mesh behavior is inherited from the BGN1 predictor.
	\end{remark}
	
	We now present the practical second-order area-preserving PFEM. Given $(\mathbf{X}^{m-1},\kappa^{m-1})$ and $(\mathbf{X}^{m},\kappa^{m})$, and the predicted curve $\widetilde{\Gamma}^{m+1}$ obtained from \eqref{eq:BGN1-predictor-2d}, find $(\mathbf{X}^{m+1},\kappa^{m+1})\in[V^h]^2\times V^h$ such that
	\begin{subequations}
		\label{eq:scheme-2d}
		\begin{align}
			&
			\frac16\left[
			\left(
			\frac{\frac32\mathbf{X}^{m+1}-2\mathbf{X}^{m}+\frac12\mathbf{X}^{m-1}}{\tau},
			\varphi^h\mathbf{n}^{m+1}
			\right)^h_{\widetilde{\Gamma}^{m+1}}
			+
			\left(
			\partial_s\kappa^{m+1},
			\partial_s\varphi^h
			\right)_{\widetilde{\Gamma}^{m+1}}
			\right]
			\nonumber\\
			&\quad
			+\frac46\left[
			\left(
			\frac{\mathbf{X}^{m+1}-\mathbf{X}^{m-1}}{2\tau},
			\varphi^h\mathbf{n}^{m}
			\right)^h_{\Gamma^{m}}
			+
			\left(
			\frac{\partial_s\kappa^{m+1}+\partial_s\kappa^{m-1}}{2},
			\partial_s\varphi^h
			\right)_{\Gamma^{m}}
			\right]
			\nonumber\\
			&\quad
			+\frac16\left[
			\left(
			\frac{-\frac12\mathbf{X}^{m+1}+2\mathbf{X}^{m}-\frac32\mathbf{X}^{m-1}}{\tau},
			\varphi^h\mathbf{n}^{m-1}
			\right)^h_{\Gamma^{m-1}}
			+
			\left(
			\partial_s\kappa^{m-1},
			\partial_s\varphi^h
			\right)_{\Gamma^{m-1}}
			\right]
			=0,
			\quad \forall \varphi^h\in V^h,
			\label{eq:scheme-2d-a}
			\\
			&
			\left(
			\kappa^{m+1},
			\widetilde{\mathbf{n}}^{m+1}\cdot\boldsymbol{\omega}^h
			\right)^h_{\widetilde{\Gamma}^{m+1}}
			-
			\left(
			\partial_s\mathbf{X}^{m+1},
			\partial_s\boldsymbol{\omega}^h
			\right)_{\widetilde{\Gamma}^{m+1}}
			=0,
			\quad
			\forall \boldsymbol{\omega}^h\in[V^h]^2.
			\label{eq:scheme-2d-b}
		\end{align}
	\end{subequations}
	We refer to \eqref{eq:scheme-2d} as the BGN/AC scheme, where AC stands for area conserving. Here $\partial_s$ is computed on the curve indicated by the subscript, $\mathbf{n}^{m}$ and $\mathbf{n}^{m-1}$ are the usual outward unit normals of $\Gamma^m$ and $\Gamma^{m-1}$, and $\widetilde{\mathbf{n}}^{m+1}:=-(\partial_\rho\widetilde{\mathbf{X}}^{m+1})^\perp/|\partial_\rho\widetilde{\mathbf{X}}^{m+1}|$ is the predictor normal used only in the curvature equation.
	
	\begin{theorem}[Exact area conservation]
		\label{thm:area-conservation-2d}
		Let $(\mathbf{X}^{m+1},\kappa^{m+1})$ be a solution of \eqref{eq:scheme-2d}. Then $A^{m+1}=A^{m-1}$. If the first step is computed by the first-order structure-preserving PFEM so that $A^1=A^0$, then $A^m=A^0$ for all $m\ge 0$.
	\end{theorem}
	
	\begin{proof}
		Taking $\varphi^h\equiv2\tau$ in \eqref{eq:scheme-2d-a}, all curvature-diffusion terms vanish. The remaining three normal velocity terms are exactly the Simpson representation of the area difference in \eqref{eq:area-identity-2d}; the only point to note is that \eqref{eq:sp-normal-2d} cancels the metric factor of $\widetilde{\Gamma}^{m+1}$. Hence $A^{m+1}-A^{m-1}=0$. The full conservation follows by induction from $A^1=A^0$.
	\end{proof}
	
	\begin{remark}[Starting step]
		The scheme \eqref{eq:scheme-2d} is a two-step method and starts from $m=1$. To obtain area conservation at all time levels, we compute the first step by the first-order structure-preserving PFEM of Bao and Zhao \cite{bao2021structure}, which gives $A^1=A^0$. For all later steps, the predicted curve is generated by the classical BGN1 scheme.
	\end{remark}
	
\begin{remark}[Compatibility with Lagrange-multiplier energy stabilization]
	The proposed BGN/AC scheme can also be combined with a Lagrange-multiplier treatment of the perimeter dissipation law, in the spirit of \cite{garcke2025structure}. At the continuous level, one may replace the normal velocity equation by
	\[
	\mathbf n\cdot\partial_t\mathbf X=(1+\lambda(t))\partial_{ss}\kappa,
	\qquad
	\kappa\mathbf n=-\partial_{ss}\mathbf X,
	\qquad
	\frac{dL}{dt}=-\int_{\Gamma(t)}|\partial_s\kappa|^2\,ds .
	\]
	and hence $\lambda(t)=0$ whenever $\partial_s\kappa\not\equiv0$.  At the fully discrete level, the multiplier only modifies the curvature-diffusion terms and therefore does not affect the proof of exact area conservation, since these terms vanish when the constant test function is used. The same observation applies to the surface case by replacing $\partial_{ss}\kappa$ with $\Delta_S H$ and the perimeter with the surface area. We do not write this variant explicitly in order to keep the presentation focused on the Simpson--Boole geometric conservation mechanism.
\end{remark}
	
	\subsection{Iterative solver for the nonlinear system}
	\label{subsec:solver-2d}
	
	The fully discrete scheme \eqref{eq:scheme-2d} is nonlinear because the structure-preserving normal
	\[
	\mathbf{n}^{m+1}
	=
	-\frac{(\partial_\rho\mathbf{X}^{m+1})^\perp}
	{|\partial_\rho\widetilde{\mathbf{X}}^{m+1}|}
	\]
	depends on the unknown curve $\mathbf{X}^{m+1}$. In this subsection, we describe the Newton-type iterative solver used in the computations. The construction follows the iterative strategy used in the first-order structure-preserving PFEM, with the only difference that the residual now contains the Simpson-weighted three-level terms.
	
	Let
	$ (\mathbf{X}^{m+1,l},\kappa^{m+1,l})$
	be the approximation at the $l$-th Newton iteration. We define
	$
	\mathbf{n}^{m+1,l}
	:=
	-\frac{(\partial_\rho\mathbf{X}^{m+1,l})^\perp}
	{|\partial_\rho\widetilde{\mathbf{X}}^{m+1}|}.
	$
	The increments are denoted by
	\[
	\delta\mathbf{X}^{l}
	:=
	\mathbf{X}^{m+1,l+1}-\mathbf{X}^{m+1,l},
	\qquad
	\delta\kappa^{l}
	:=
	\kappa^{m+1,l+1}-\kappa^{m+1,l},
	\qquad
	\delta\mathbf{n}^{l}
	=
	-\frac{(\partial_\rho\delta\mathbf{X}^{l})^\perp}
	{|\partial_\rho\widetilde{\mathbf{X}}^{m+1}|}.
	\]

	At each Newton iteration, we solve for
	$(\delta\mathbf{X}^{l},\delta\kappa^{l})\in [V^h]^2\times V^h$
	such that, for all $\varphi^h\in V^h$ and $\boldsymbol{\omega}^h\in[V^h]^2$,
	\begin{subequations}
		\label{eq:newton-2d}
		\begin{align}
			&
			\frac16
			\left[
			\left(
			\frac{3}{2\tau}\delta\mathbf{X}^{l},
			\varphi^h\mathbf{n}^{m+1,l}
			\right)^h_{\widetilde{\Gamma}^{m+1}}
			+
			\left(
			\frac{\frac32\mathbf{X}^{m+1,l}-2\mathbf{X}^{m}+\frac12\mathbf{X}^{m-1}}{\tau},
			\varphi^h\delta\mathbf{n}^{l}
			\right)^h_{\widetilde{\Gamma}^{m+1}}
			+
			\left(
			\partial_s\delta\kappa^{l},
			\partial_s\varphi^h
			\right)_{\widetilde{\Gamma}^{m+1}}
			\right]
			\nonumber\\
			&\quad
			+\frac46
			\left[
			\left(
			\frac{\delta\mathbf{X}^{l}}{2\tau},
			\varphi^h\mathbf{n}^{m}
			\right)^h_{\Gamma^{m}}
			+
			\left(
			\frac{\partial_s\delta\kappa^{l}}{2},
			\partial_s\varphi^h
			\right)_{\Gamma^{m}}
			\right]
			+\frac16
			\left[
			\left(
			-\frac{\delta\mathbf{X}^{l}}{2\tau},
			\varphi^h\mathbf{n}^{m-1}
			\right)^h_{\Gamma^{m-1}}
			\right]
			=
			-\mathcal{R}_1^{l}(\varphi^h),
			\label{eq:newton-2d-a}
			\\
			&
			\left(
			\delta\kappa^{l},
			\widetilde{\mathbf{n}}^{m+1}\cdot\boldsymbol{\omega}^h
			\right)^h_{\widetilde{\Gamma}^{m+1}}
			-
			\left(
			\partial_s\delta\mathbf{X}^{l},
			\partial_s\boldsymbol{\omega}^h
			\right)_{\widetilde{\Gamma}^{m+1}}
			=
			-\mathcal{R}_2^{l}(\boldsymbol{\omega}^h).
			\label{eq:newton-2d-b}
		\end{align}
	\end{subequations}
	Here $\mathcal{R}_1^l$ and $\mathcal{R}_2^l$ are the residuals of \eqref{eq:scheme-2d-a} and \eqref{eq:scheme-2d-b} evaluated at 
	$(\mathbf{X}^{m+1,l},\kappa^{m+1,l})$. More precisely,
	\begin{align}
		\mathcal{R}_1^{l}(\varphi^h)
		:={}&
		\frac16\left[
		\left(
		\frac{\frac32\mathbf{X}^{m+1,l}-2\mathbf{X}^{m}+\frac12\mathbf{X}^{m-1}}{\tau},
		\varphi^h\mathbf{n}^{m+1,l}
		\right)^h_{\widetilde{\Gamma}^{m+1}}
		+
		\left(
		\partial_s\kappa^{m+1,l},
		\partial_s\varphi^h
		\right)_{\widetilde{\Gamma}^{m+1}}
		\right]
		\nonumber\\
		&+
		\frac46\left[
		\left(
		\frac{\mathbf{X}^{m+1,l}-\mathbf{X}^{m-1}}{2\tau},
		\varphi^h\mathbf{n}^{m}
		\right)^h_{\Gamma^{m}}
		+
		\left(
		\frac{\partial_s\kappa^{m+1,l}+\partial_s\kappa^{m-1}}{2},
		\partial_s\varphi^h
		\right)_{\Gamma^{m}}
		\right]
		\nonumber\\
		&+
		\frac16\left[
		\left(
		\frac{-\frac12\mathbf{X}^{m+1,l}+2\mathbf{X}^{m}-\frac32\mathbf{X}^{m-1}}{\tau},
		\varphi^h\mathbf{n}^{m-1}
		\right)^h_{\Gamma^{m-1}}
		+
		\left(
		\partial_s\kappa^{m-1},
		\partial_s\varphi^h
		\right)_{\Gamma^{m-1}}
		\right],
		\label{eq:R1-2d}
	\end{align}
	and
	\begin{equation}
		\mathcal{R}_2^{l}(\boldsymbol{\omega}^h)
		:=
		\left(
		\kappa^{m+1,l},
		\widetilde{\mathbf{n}}^{m+1}\cdot\boldsymbol{\omega}^h
		\right)^h_{\widetilde{\Gamma}^{m+1}}
		-
		\left(
		\partial_s\mathbf{X}^{m+1,l},
		\partial_s\boldsymbol{\omega}^h
		\right)_{\widetilde{\Gamma}^{m+1}}.
		\label{eq:R2-2d}
	\end{equation}
	After solving \eqref{eq:newton-2d}, we update
	\[
	\mathbf{X}^{m+1,l+1}
	=
	\mathbf{X}^{m+1,l}+\delta\mathbf{X}^{l},
	\qquad
	\kappa^{m+1,l+1}
	=
	\kappa^{m+1,l}+\delta\kappa^{l}.
	\]
	The iteration is stopped once
	\[
	\max\left\{
	\|\delta\mathbf{X}^{l}\|_{\ell^\infty},
	\|\delta\kappa^{l}\|_{\ell^\infty}
	\right\}
	\le \mathrm{tol}.
	\]
	
	The initial guess is chosen as
	$
	\mathbf{X}^{m+1,0}=\widetilde{\mathbf{X}}^{m+1},
	\kappa^{m+1,0}=\kappa^m,
	$
	where $\widetilde{\mathbf{X}}^{m+1}$ is obtained by one BGN1 predictor step. This choice is important in practice: it reduces the number of Newton iterations and inherits the good mesh distribution property of the classical BGN1 scheme.
	
	
	\begin{remark}
		In \eqref{eq:newton-2d-b}, we have used the predictor normal
		\[
		\widetilde{\mathbf{n}}^{m+1}
		=
		-\frac{(\partial_\rho\widetilde{\mathbf{X}}^{m+1})^\perp}
		{|\partial_\rho\widetilde{\mathbf{X}}^{m+1}|}.
		\]
		This is consistent with the practical implementation. The exact area conservation proof only uses the structure-preserving normal vector in the normal velocity equation \eqref{eq:scheme-2d-a}. Therefore, using the predictor normal in the curvature equation \eqref{eq:scheme-2d-b} does not affect the exact area conservation, while it reduces the nonlinearity of the Newton system.
	\end{remark}
	
\section{A second-order volume-preserving PFEM for closed surfaces in 3D}
\label{sec:3D}

We now extend the construction of Section~\ref{sec:2D} to the surface diffusion of closed surfaces in $\mathbb{R}^3$. The presentation parallels the two-dimensional case, but the exact conservation identity now involves the oriented area vector of a surface patch and therefore requires a higher-order quadrature formula.

Let $S(t)$ be a closed evolving surface parametrized by
$
\mathbf{X}(\cdot,t):S_0\to\mathbb{R}^3,
$
where $S_0$ is a fixed smooth reference surface. The surface diffusion flow reads
\begin{equation}
	\mathbf{n}\cdot\partial_t\mathbf{X}=\Delta_S H,
	\qquad
	H\mathbf{n}=-\Delta_S\mathbf{X},
	\label{eq:strong-3d}
\end{equation}
where $H$ is the scalar mean curvature, $\mathbf{n}$ is the outward unit normal vector, and $\Delta_S$ is the Laplace--Beltrami operator on $S(t)$. The enclosed volume $V(t)$ and the total surface area $A(t)$ satisfy
\[
\frac{d}{dt}V(t)=0,\qquad
\frac{d}{dt}A(t)=-\int_{S(t)}|\nabla_S H|^2\,dA\le 0.
\]

\subsection{Weak formulation and spatial discretization}

For notational simplicity, scalar and vector fields on $S(t)$ are identified with their pullbacks to the reference surface $S_0$ through the parametrization $\mathbf{X}(\cdot,t)$. With this convention, the weak formulation reads: find $\mathbf{X}(\cdot,t)\in[H^1(S_0)]^3$ and $H(\cdot,t)\in H^1(S_0)$ such that
\begin{subequations}
	\begin{align}
		\left(
		\mathbf{n}\cdot\partial_t\mathbf{X},
		\varphi
		\right)_{S(t)}
		+
		\left(
		\nabla_{S(t)}H,
		\nabla_{S(t)}\varphi
		\right)_{S(t)}
		&=0,
		\qquad \forall \varphi\in H^1(S_0),
		\label{eq:weak-3d-a}
		\\
		\left(
		H,
		\mathbf{n}\cdot\boldsymbol{\omega}
		\right)_{S(t)}
		-
		\left(
		\nabla_{S(t)}\mathbf{X},
		\nabla_{S(t)}\boldsymbol{\omega}
		\right)_{S(t)}
		&=0,
		\qquad \forall \boldsymbol{\omega}\in[H^1(S_0)]^3.
		\label{eq:weak-3d-b}
	\end{align}
\end{subequations}
Here
\[
\left(
\nabla_{S(t)}\mathbf{X},
\nabla_{S(t)}\boldsymbol{\omega}
\right)_{S(t)}
:=
\sum_{\ell=1}^3
\left(
\nabla_{S(t)}X_\ell,
\nabla_{S(t)}\omega_\ell
\right)_{S(t)}.
\]

We now introduce the discrete notation used below. Let
$
\widehat S_h=\bigcup_{j=1}^J \widehat\sigma_j
$
be a fixed triangulated approximation of $S_0$ with no hanging vertices, where each $\widehat\sigma_j$ is a nondegenerate triangle. The discrete surfaces at different time levels are represented by nodal maps
\[
\mathbf{X}^m\in[\mathcal V^h]^3,\qquad
S^m=\mathbf{X}^m(\widehat S_h)
=
\bigcup_{j=1}^J \sigma_j^m,
\]
where
\[
\sigma_j^m
=
\triangle\{\mathbf q_{j1}^m,\mathbf q_{j2}^m,\mathbf q_{j3}^m\},
\qquad
\mathbf q_{jk}^m:=\mathbf{X}^m(\widehat{\mathbf q}_{jk}).
\]
The vertices of each triangle are ordered so that the induced normal is outward. The scalar finite element space on the reference triangulation is
\[
\mathcal V^h
:=
\left\{
u\in C(\widehat S_h):\
u|_{\widehat\sigma_j}\in\mathbb P_1,\quad j=1,\ldots,J
\right\}.
\]

For a nodal map $\mathbf Z\in[\mathcal V^h]^3$, we write
\[
S_{\mathbf Z}:=\mathbf Z(\widehat S_h)
\]
for the corresponding discrete surface. Functions in $\mathcal V^h$ are naturally identified with continuous piecewise linear functions on $S_{\mathbf Z}$ through the same nodal values and mesh connectivity. This convention is used throughout the 3D scheme when a quantity such as $H^{m+1}$ is assembled on a predicted surface $\widetilde S^{m+1}$.

For a triangle
\[
\sigma_j(\mathbf Z)
=
\triangle\{\mathbf z_{j1},\mathbf z_{j2},\mathbf z_{j3}\},
\qquad
\mathbf z_{jk}:=\mathbf Z(\widehat{\mathbf q}_{jk}),
\]
define its oriented area vector by
\begin{equation}
	\mathcal J_j(\mathbf Z)
	:=
	(\mathbf z_{j2}-\mathbf z_{j1})
	\times
	(\mathbf z_{j3}-\mathbf z_{j1}).
	\label{eq:discrete-J-3d}
\end{equation}
Then
$
|\sigma_j(\mathbf Z)|
=
\frac12|\mathcal J_j(\mathbf Z)|,
\mathbf n_j(\mathbf Z)
=
\frac{\mathcal J_j(\mathbf Z)}
{|\mathcal J_j(\mathbf Z)|}.
$
By abuse of notation, $\mathcal J(\mathbf Z)$ and $\mathbf n(\mathbf Z)$ also denote the corresponding piecewise constant vector fields on $\widehat S_h$ and $S_{\mathbf Z}$, respectively, with
$
\mathcal J(\mathbf Z)|_{\widehat\sigma_j}
=
\mathcal J_j(\mathbf Z),
\mathbf n(\mathbf Z)|_{\sigma_j(\mathbf Z)}
=
\mathbf n_j(\mathbf Z).
$

For scalar or vector-valued piecewise continuous functions $u,v$ on $S_{\mathbf Z}$, we define
$
(u,v)_{S_{\mathbf Z}}
:=
\int_{S_{\mathbf Z}} u\cdot v\,dA,
$
and the mass-lumped inner product
\begin{equation}
	(u,v)^h_{S_{\mathbf Z}}
	:=
	\frac13
	\sum_{j=1}^J
	|\sigma_j(\mathbf Z)|
	\sum_{k=1}^3
	(u\cdot v)(\mathbf z_{jk}).
	\label{eq:mass-lumped-3d}
\end{equation}

The discrete surface gradient is constant on each triangle. For
$f\in\mathcal V^h$, its surface gradient on $\sigma_j(\mathbf Z)$ is defined by
\begin{equation}
	\begin{aligned}
		\left(\nabla_{S_{\mathbf Z}} f\right)|_{\sigma_j(\mathbf Z)}
		=
		\frac{
			\left[
			f_{j1}(\mathbf z_{j3}-\mathbf z_{j2})
			+f_{j2}(\mathbf z_{j1}-\mathbf z_{j3})
			+f_{j3}(\mathbf z_{j2}-\mathbf z_{j1})
			\right]
			\times \mathbf n_j(\mathbf Z)
		}
		{|\mathcal J_j(\mathbf Z)|},
	\end{aligned}
	\label{eq:surface-gradient-3d}
\end{equation}
where $f_{jk}:=f(\widehat{\mathbf q}_{jk})$.

For later use, the discrete surface area and enclosed volume of $S_{\mathbf Z}$ are
defined by
\begin{equation}
	A(\mathbf Z)
	:=
	\sum_{j=1}^J |\sigma_j(\mathbf Z)|,
	\qquad
	V(\mathbf Z)
	:=
	\frac{1}{18}
	\sum_{j=1}^J
	\sum_{k=1}^3
	\mathbf z_{jk}\cdot \mathcal J_j(\mathbf Z).
	\label{eq:discrete-area-volume-3d}
\end{equation}
In particular, $A^m:=A(\mathbf X^m)$ and $V^m:=V(\mathbf X^m)$. We assume throughout that all discrete and predicted surfaces used in the scheme are nondegenerate, i.e.,
$
\min_{1\le j\le J}|\mathcal J_j(\mathbf Z)|>0
$
for the corresponding nodal maps $\mathbf Z$.

	\subsection{Quadratic interpolation and Boole's formula}
	
	As in the two-dimensional case, define
	\begin{equation}
		\begin{aligned}
			\mathbf{X}^h(\alpha)
			&=
			(2\alpha^2-3\alpha+1)\mathbf{X}^{m-1}
			+(-4\alpha^2+4\alpha)\mathbf{X}^{m}
			+(2\alpha^2-\alpha)\mathbf{X}^{m+1},
			\qquad 0\le\alpha\le1.
		\end{aligned}
		\label{eq:quad-Y-3d}
	\end{equation}
	The five quadrature points are
	\[
	\alpha=0,\quad \frac14,\quad \frac12,\quad \frac34,\quad 1,
	\]
	corresponding to
	\[
	t_{m-1},\quad t_{m-\frac12},\quad t_m,\quad t_{m+\frac12},\quad t_{m+1}.
	\]
	The corresponding interpolated positions are
	\begin{align}
		&\mathbf{X}^{m-1} := \mathbf{X}^h(0), \quad \mathbf{X}^{m} := \mathbf{X}^h\!\left(\frac12\right), \quad \mathbf{X}^{m+1} := \mathbf{X}^h(1),
		\\
		&\mathbf{X}^{m-\frac12}
		:=\mathbf{X}^h\!\left(\frac14\right)
		=
		\frac38\mathbf{X}^{m-1}
		+\frac34\mathbf{X}^{m}
		-\frac18\mathbf{X}^{m+1},\\
		&\mathbf{X}^{m+\frac12}
		:=
		\mathbf{X}^h\!\left(\frac34\right)
		=
		-\frac18\mathbf{X}^{m-1}
		+\frac34\mathbf{X}^{m}
		+\frac38\mathbf{X}^{m+1}.
		\label{eq:Y-half}
	\end{align}
	The endpoint and midpoint velocity approximations are the same as those in \eqref{eq:D-mminus}--\eqref{eq:D-mplus}. The two additional Boole quadrature points give
	\begin{align}
		\partial_t\mathbf X^{m-\frac12}
		&:=
		\frac{\mathbf X^{m}-\mathbf X^{m-1}}{\tau},
		\label{eq:D-3d-mhalf}
		\\
		\partial_t\mathbf X^{m+\frac12}
		&:=
		\frac{\mathbf X^{m+1}-\mathbf X^{m}}{\tau}.
		\label{eq:D-3d-phalf}
	\end{align}
	These half-point values give the predictor-corrector half-step velocities used in the practical scheme below.
	
	For the interpolated surface, we use the same notation $\mathcal J(\mathbf X^h(\alpha))$ for the oriented area vector defined elementwise as in \eqref{eq:discrete-J-3d}. Let $V(\alpha)$ be the volume enclosed by $S^h(\alpha)=\mathbf X^h(\alpha)(S_0)$. Then
	\begin{equation}
		\begin{aligned}
			\frac{d}{d\alpha}V(\alpha)
			=
			\int_{S_0}
			\partial_\alpha\mathbf{X}^h(\alpha)\cdot
			\mathcal{J}(\mathbf{X}^h(\alpha))\,dA_0.
		\end{aligned}
		\label{eq:dV-dalpha}
	\end{equation}
	Here $\partial_\alpha\mathbf X^h$ is linear in $\alpha$. Moreover, for each fixed reference triangle, the two edge vectors entering $\mathcal J(\mathbf X^h(\alpha))$ are quadratic polynomials in $\alpha$, and hence their cross product is of degree at most four. Therefore the oriented volume density in \eqref{eq:dV-dalpha} is a polynomial of degree at most five in $\alpha$. Boole's rule is therefore exact, and we obtain
	\begin{equation}
		\begin{aligned}
			V^{m+1}-V^{m-1}= & \int_0^1 \frac{\mathrm{~d}}{\mathrm{~d} \alpha} V(\alpha) \mathrm{d} \alpha \\
			= & \int_0^1 \int_{S^h(\alpha)}\partial_\alpha \mathbf{X}^h(\alpha) \cdot \boldsymbol{n}^h(\alpha) \mathrm{d} A \mathrm{~d} \alpha \\
			= & 
			\int_{\widehat S_h}\frac{1}{90}(7[-\mathbf{X}^{m+1}+4\mathbf{X}^m-3\mathbf{X}^{m-1}] \cdot \mathcal{J}(\mathbf{X}^{m-1})+32[2\mathbf{X}^m-2\mathbf{X}^{m-1}] \cdot \\ &\mathcal{J}(\frac{3}{8}\mathbf{X}^{m-1} + \frac{3}{4}\mathbf{X}^{m} - \frac{1}{8}\mathbf{X}^{m+1}) + 12[\mathbf{X}^{m+1}-\mathbf{X}^{m-1}] \cdot \mathcal{J}(\mathbf{X}^m) + \\ & 32[2\mathbf{X}^{m+1}-2\mathbf{X}^{m}] \cdot \mathcal{J}(-\frac{1}{8}\mathbf{X}^{m-1} + \frac{3}{4}\mathbf{X}^m + \frac{3}{8}\mathbf{X}^{m+1}) + \\ & 7[3\mathbf{X}^{m+1}-4\mathbf{X}^m + \mathbf{X}^{m-1}] \cdot \mathcal{J}(\mathbf{X}^{m+1})) \mathrm{d} A_0.
		\end{aligned}
		\label{eq:volume-identity-3d}
	\end{equation}
	
	\begin{remark}
		The two half-time levels $m-\frac12$ and $m+\frac12$ arise from Boole's formula and have no analogue in the two-dimensional Simpson identity. In the practical scheme below, the corresponding geometries are not treated as additional unknowns; they are supplied by BGN1 predictor steps, in the spirit of predictor--corrector BGN methods \cite{jiang2025predictor}. The volume conservation mechanism is retained because the normal vectors are defined through the oriented area vectors associated with the quadratic interpolation path.
	\end{remark}
	
	\subsection{Prediction and the second-order volume-preserving scheme}
	
	 Compared with the two-dimensional case, the main additional feature is that Boole's formula involves two half-time geometries. The predictor $\widetilde{S}^{m+1}$ is obtained by one BGN1 step from $S^m$ with time step $\tau$, while the half-time predictor $\widetilde{S}^{m+\frac12}$ is obtained by one BGN1 step from $S^m$ with time step $\tau/2$. For $m\ge2$, the predictor $\widetilde{S}^{m-\frac12}$ has already been generated in the previous step; for the first two-step update, $\widetilde{S}^{\frac12}$ is generated during the initialization by a BGN1 half step from $S^0$. 
	
	More precisely, $\widetilde{S}^{m+1}$ is obtained by finding $\widetilde{\mathbf{X}}^{m+1}\in[\mathcal V^h]^3$ and $\widetilde{H}^{m+1}\in\mathcal V^h$ such that
	\begin{subequations}
		\label{eq:BGN1-predictor-3d-full}
		\begin{align}
			\left(
			\frac{\widetilde{\mathbf{X}}^{m+1}-\mathbf{X}^{m}}{\tau},
			\varphi^h\mathbf{n}^{m}
			\right)^h_{S^{m}}
			+
			\left(
			\nabla_{S^m}\widetilde{H}^{m+1},
			\nabla_{S^m}\varphi^h
			\right)_{S^m}
			&=0,
			\qquad \forall \varphi^h\in\mathcal V^h,
			\\
			\left(
			\widetilde{H}^{m+1},
			\mathbf{n}^{m}\cdot\boldsymbol{\omega}^h
			\right)^h_{S^{m}}
			-
			\left(
			\nabla_{S^m}\widetilde{\mathbf{X}}^{m+1},
			\nabla_{S^m}\boldsymbol{\omega}^h
			\right)_{S^m}
			&=0,
			\qquad \forall \boldsymbol{\omega}^h\in[\mathcal V^h]^3.
		\end{align}
	\end{subequations}
	The half-time predictor $\widetilde{S}^{m+\frac12}$ is defined analogously by replacing $\tau$ in \eqref{eq:BGN1-predictor-3d-full} with $\tau/2$.
	
	Based on the exact identity \eqref{eq:volume-identity-3d} , the weighted structure-preserving normal vectors are defined by
	\begin{align}
		\mathbf{n}^{m+1}
		&:=
		\frac{\mathcal{J}(\mathbf{X}^{m+1})}
		{|\mathcal{J}(\widetilde{\mathbf{X}}^{m+1})|},
		\label{eq:n-mplus-3d}
		\\
		\mathbf{n}^{m+\frac12}
		&:=
		\frac{\mathcal{J}\left(-\frac18\mathbf{X}^{m-1}+\frac34\mathbf{X}^{m}+\frac38\mathbf{X}^{m+1}\right)}
		{|\mathcal{J}(\widetilde{\mathbf{X}}^{m+\frac12})|},
		\label{eq:n-phalf-3d}
		\\
		\mathbf{n}^{m-\frac12}
		&:=
		\frac{\mathcal{J}\left(\frac38\mathbf{X}^{m-1}+\frac34\mathbf{X}^{m}-\frac18\mathbf{X}^{m+1}\right)}
		{|\mathcal{J}(\widetilde{\mathbf{X}}^{m-\frac12})|}.
		\label{eq:n-mhalf-3d}
	\end{align}
	At the known time levels $m$ and $m-1$, we use the usual outward normals $\mathbf{n}^{m}=\mathcal{J}(\mathbf{X}^{m})/|\mathcal{J}(\mathbf{X}^{m})|$ and $\mathbf{n}^{m-1}=\mathcal{J}(\mathbf{X}^{m-1})/|\mathcal{J}(\mathbf{X}^{m-1})|$. The definitions above ensure that integration on the predicted surfaces gives the same oriented volume densities as those in the exact identity \eqref{eq:volume-identity-3d}.
	
	We are now ready to state the practical second-order volume-preserving PFEM. Given $(\mathbf{X}^{m-1},H^{m-1})$ and $(\mathbf{X}^{m},H^{m})$, find
	$(\mathbf{X}^{m+1},H^{m+1})\in[\mathcal V^h]^3\times\mathcal V^h$ such that, for all
	$(\varphi^h,\boldsymbol{\omega}^h)\in\mathcal V^h\times[\mathcal V^h]^3$,
	\begin{subequations}
		\label{eq:scheme-3d}
		\begin{align}
			&
			\frac{7}{90}\left[
			\left(
			\frac{3\mathbf{X}^{m+1}-4\mathbf{X}^{m}+\mathbf{X}^{m-1}}{2\tau},
			\varphi^h\mathbf{n}^{m+1}
			\right)^h_{\widetilde{S}^{m+1}}
			+
			\left(
			\nabla_{\widetilde{S}^{m+1}}H^{m+1},
			\nabla_{\widetilde{S}^{m+1}}\varphi^h
			\right)_{\widetilde{S}^{m+1}}
			\right]
			\nonumber\\
			&\quad
			+\frac{32}{90}\left[
			\left(
			\frac{\mathbf{X}^{m+1}-\mathbf{X}^{m}}{\tau},
			\varphi^h\mathbf{n}^{m+\frac12}
			\right)^h_{\widetilde{S}^{m+\frac12}}
			+
			\left(
			\nabla_{\widetilde{S}^{m+\frac12}}\frac{H^{m+1}+H^{m}}{2},
			\nabla_{\widetilde{S}^{m+\frac12}}\varphi^h
			\right)_{\widetilde{S}^{m+\frac12}}
			\right]
			\nonumber\\
			&\quad
			+\frac{12}{90}\left[
			\left(
			\frac{\mathbf{X}^{m+1}-\mathbf{X}^{m-1}}{2\tau},
			\varphi^h\mathbf{n}^{m}
			\right)^h_{S^{m}}
			+
			\left(
			\nabla_{S^{m}}H^{m},
			\nabla_{S^{m}}\varphi^h
			\right)_{S^{m}}
			\right]
			\nonumber\\
			&\quad
			+\frac{32}{90}\left[
			\left(
			\frac{\mathbf{X}^{m}-\mathbf{X}^{m-1}}{\tau},
			\varphi^h\mathbf{n}^{m-\frac12}
			\right)^h_{\widetilde{S}^{m-\frac12}}
			+
			\left(
			\nabla_{\widetilde{S}^{m-\frac12}}\frac{H^{m}+H^{m-1}}{2},
			\nabla_{\widetilde{S}^{m-\frac12}}\varphi^h
			\right)_{\widetilde{S}^{m-\frac12}}
			\right]
			\nonumber\\
			&\quad
			+\frac{7}{90}\left[
			\left(
			\frac{-\mathbf{X}^{m+1}+4\mathbf{X}^{m}-3\mathbf{X}^{m-1}}{2\tau},
			\varphi^h\mathbf{n}^{m-1}
			\right)^h_{S^{m-1}}
			+
			\left(
			\nabla_{S^{m-1}}H^{m-1},
			\nabla_{S^{m-1}}\varphi^h
			\right)_{S^{m-1}}
			\right]
			=0,
			\quad \forall \varphi^h\in\mathcal V^h,
			\label{eq:scheme-3d-a}
			\\
			&
			\left(
			H^{m+1},
			\widetilde{\mathbf{n}}^{m+1}\cdot\boldsymbol{\omega}^h
			\right)^h_{\widetilde{S}^{m+1}}
			-
			\left(
			\nabla_{\widetilde{S}^{m+1}}\mathbf{X}^{m+1},
			\nabla_{\widetilde{S}^{m+1}}\boldsymbol{\omega}^h
			\right)_{\widetilde{S}^{m+1}}
			=0,
			\quad \forall \boldsymbol{\omega}^h\in[\mathcal V^h]^3.
			\label{eq:scheme-3d-b}
		\end{align}
	\end{subequations}
	We refer to \eqref{eq:scheme-3d} as the BGN/VC scheme, where VC stands for volume conserving. In \eqref{eq:scheme-3d}, $\widetilde{\mathbf{n}}^{m+1}$ is the usual outward normal of the predicted surface $\widetilde{S}^{m+1}$ and is used only in the curvature equation. The nonlinear system is solved by a Newton-type method analogous to that in Section~\ref{subsec:solver-2d}; the only additional ingredient is the linearization of the oriented area vector $\mathcal{J}(\mathbf{X})$. The details are given in Appendix~\ref{app:solver-3d}.
	
	\begin{theorem}[Exact volume conservation]
		\label{thm:volume-conservation-3d}
		Let $(\mathbf{X}^{m+1},H^{m+1})$ be a solution of \eqref{eq:scheme-3d}. Then $V^{m+1}=V^{m-1}$. If the first step is computed by the first-order structure-preserving PFEM so that $V^1=V^0$, then $V^m=V^0$ for all $m\ge0$.
	\end{theorem}
	
	\begin{proof}
		Taking $\varphi^h\equiv2\tau$ in \eqref{eq:scheme-3d-a}, all surface-gradient terms vanish. The remaining five normal velocity terms are exactly the Boole representation of the volume difference in \eqref{eq:volume-identity-3d}; the only point to note is that the definitions \eqref{eq:n-mplus-3d}--\eqref{eq:n-mhalf-3d} convert the area factors of the predicted surfaces into the oriented area vectors appearing in \eqref{eq:volume-identity-3d}. Hence $V^{m+1}-V^{m-1}=0$. The full conservation follows by induction from $V^1=V^0$.
	\end{proof}
	
	\begin{remark}
		The five terms in \eqref{eq:scheme-3d-a} correspond exactly to the five Boole quadrature points. The endpoint velocity approximations are BDF2-type, the center term is of CN-leapfrog type, and the two half-time terms are predictor-corrector-type. Thus the scheme keeps the predictor-based implementation of existing second-order BGN methods while replacing the normal densities by the structure-preserving ones dictated by the exact volume identity.
	\end{remark}
	
	\section{Numerical results}
	\label{sec:numerics}
	
	In this section, we present numerical experiments to examine the accuracy, conservation properties, geometric behavior, and computational performance of the proposed second-order schemes. Both the curve case in two dimensions and the surface case in three dimensions are considered. The nonlinear systems are solved by Newton's method with stopping tolerance $\mathrm{tol}=5\times10^{-14}$.
	
	After the initial polygonal curve or polyhedral surface represented by $\mathbf{X}^0$ has been generated, the initial curvature variable, namely $\kappa^0$ in 2D or $H^0$ in 3D, is computed by the least-squares procedure described in \cite{jiang2024second}.
	
	
	For the convergence tests of closed curves, we compare numerical solutions obtained on two successive refinements. The discrepancy between two curves is measured by the manifold distance. More precisely, for two closed curves $\Gamma_1$ and $\Gamma_2$ enclosing regions $\Omega_1$ and $\Omega_2$, respectively, we define
	\[
	M(\Gamma_1,\Gamma_2)
	:=
	|\Omega_1\triangle\Omega_2|
	=
	|(\Omega_1\setminus\Omega_2)\cup(\Omega_2\setminus\Omega_1)|.
	\]
	
	For 2D computations, we monitor the relative area loss and the mesh ratio:
	\[
	\Delta A(t)|_{t=t_m}:= \frac{A^m-A^0}{A^0},\qquad
	\Psi(t)|_{t=t_m}:=\frac{\max_{1\le j\le N}|\mathbf{h}_j^m|}
	{\min_{1\le j\le N}|\mathbf{h}_j^m|}.
	\]
	Here $A^m$ is the enclosed area, $\mathbf{h}_j^m$ denotes the $j$-th edge of the polygonal curve $\Gamma^m$.
	
	For 3D computations, we monitor the relative volume loss and the triangle-area ratio:
	\[
	\Delta V(t)|_{t=t_m}:= \frac{V^m-V^0}{V^0},\qquad
	r_a(t)|_{t=t_m}
	:=
	\frac{\max_j |\sigma_j^m|}{\min_j |\sigma_j^m|}.
	\]

	\subsection{Curve experiments}

We first examine the proposed BGN/AC scheme for the surface diffusion of closed planar curves. Three representative initial curves are considered.

\paragraph{Ellipse.}
\[
x=2\cos\theta,\qquad y=\sin\theta,\qquad 0\le\theta<2\pi.
\]

\paragraph{Flower curve.}
\[
x=(2+\cos(6\theta))\cos\theta,\qquad
y=(2+\cos(6\theta))\sin\theta,\qquad 0\le\theta<2\pi.
\]

\paragraph{Peanut-shaped curve.}
\[
x=(1+0.45\cos(2\theta))\cos\theta,\qquad
y=(1+0.45\cos(2\theta))\sin\theta,\qquad 0\le\theta<2\pi.
\]

The ellipse is used for the convergence study and for comparison with existing second-order BGN-based schemes. The flower and peanut-shaped curves are used to test the behavior of the method for geometries with stronger curvature variation and less uniform mesh distributions.

We begin with a convergence test for the $2\times1$ ellipse. Since an exact solution is not available, the error is evaluated by comparing two numerical solutions on successive refinements along the path $\tau=0.1h$. Specifically, for two successive time step sizes $\tau_1$ and $\tau_2$, we define
\[
\mathcal{E}_{M}(T,\tau_1,\tau_2)
:=
M\!\left(\Gamma^{T/\tau_1}_{h_1,\tau_1},\Gamma^{T/\tau_2}_{h_2,\tau_2}\right),
\]
and compute the observed order by
\[
\mathrm{Order}
:=
\log\!\left(
\frac{\mathcal{E}_{M}(T,\tau_1,\tau_2)}
{\mathcal{E}_{M}(T,\tau_2,\tau_3)}
\right)
\Big/
\log\!\left(\frac{\tau_1}{\tau_2}\right).
\]
The results are reported in Table~\ref{tab:2d-convergence}. The observed rates are close to two at all listed times, confirming the expected second-order behavior.

\begin{table}[H]
	\centering
	\renewcommand{\arraystretch}{1.2}
	\caption{Cauchy-type convergence test for the proposed BGN/AC scheme applied to the surface diffusion of a $2\times1$ ellipse. The refinement path is chosen as $\tau=0.1h$.}
	\label{tab:2d-convergence}
	\begin{tabular}{c|cc|cc|cc}
		\hline
		$\tau$
		& $\mathcal{E}_{M}(0.05)$ & order
		& $\mathcal{E}_{M}(0.5)$ & order
		& $\mathcal{E}_{M}(2)$ & order\\
		\hline
		$1/160$  & $1.40\times10^{-1}$ & --   & $1.39\times10^{-1}$ & --   & $1.39\times10^{-1}$ & --   \\
		$1/320$  & $3.60\times10^{-2}$ & $1.96$ & $3.57\times10^{-2}$ & $1.96$ & $3.57\times10^{-2}$ & $1.96$ \\
		$1/640$  & $9.05\times10^{-3}$ & $1.99$ & $9.05\times10^{-3}$ & $1.98$ & $9.00\times10^{-3}$ & $1.99$ \\
		$1/1280$ & $2.27\times10^{-3}$ & $2.00$ & $2.29\times10^{-3}$ & $1.98$ & $2.27\times10^{-3}$ & $1.99$ \\
		$1/2560$ & $5.67\times10^{-4}$ & $2.00$ & $5.75\times10^{-4}$ & $2.00$ & $5.68\times10^{-4}$ & $2.00$ \\
		\hline
	\end{tabular}
\end{table}

We next compare the proposed BGN/AC scheme with three existing second-order BGN-based methods: the BGN/CNLF scheme, the BGN/BDF2 scheme, and the BGN/PC scheme. The initial curve is again the $2\times1$ ellipse, with $
N=128, \tau=10^{-3}, T=0.5.$
The results are shown in Fig.~\ref{fig:ellipse-2d-comparison}. All four schemes decrease the perimeter during the evolution; see Fig.~\ref{fig:ellipse-2d-comparison}(a). The difference lies in the area behavior. As shown in Fig.~\ref{fig:ellipse-2d-comparison}(b), the proposed BGN/AC scheme preserves the enclosed area up to the nonlinear solver tolerance, whereas the other three second-order BGN-based schemes exhibit area drift of order $10^{-5}$. Fig.~\ref{fig:ellipse-2d-comparison}(c) further shows that the mesh ratio of the proposed method remains well controlled and is comparable to that of the other BGN-based schemes. This example indicates that the conservative modification enforces exact area preservation without sacrificing the mesh quality typical of BGN-type discretizations.

\begin{figure}[htbp]
	\centering
	\includegraphics[width=1\linewidth]{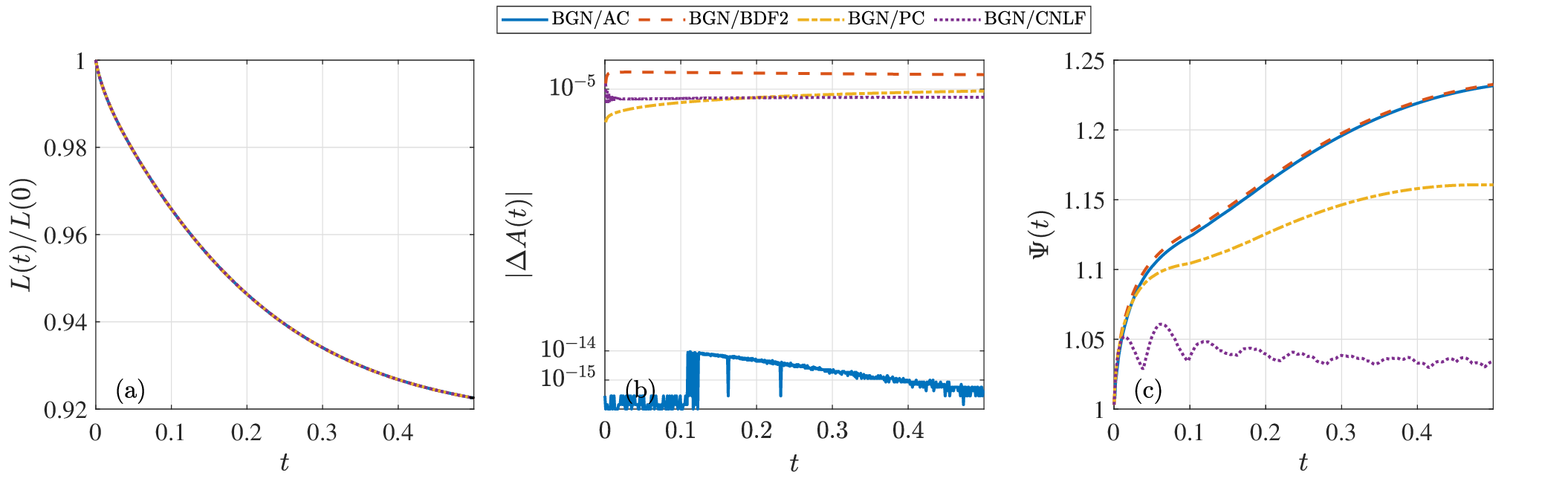}
	\caption{Comparison of four second-order BGN-based schemes for the surface diffusion of a $2\times1$ ellipse with $N=128$, $\tau=1/1000$, and $T=0.5$: (a) normalized perimeter $L(t)/L(0)$; (b) absolute relative area loss $|\Delta A(t)|$; (c) mesh ratio function $\Psi(t)$. The four curves correspond to the proposed BGN/AC scheme, the BGN/CNLF scheme, the BGN/BDF2 scheme, and the BGN/PC scheme, respectively.}
	\label{fig:ellipse-2d-comparison}
\end{figure}

We further test the method on the flower curve and the peanut-shaped curve. These two examples are more demanding than the ellipse because of their stronger curvature variation and nonuniform initial point distribution. For each initial curve, we use two time step sizes, $\tau=1/500$ and $\tau=1/1000$, to examine the sensitivity of the geometric diagnostics to temporal refinement. As shown in Fig.~\ref{fig:flower-peanut-2d}(a, e), both curves relax toward a circular equilibrium shape during the evolution. The relative area loss in Fig.~\ref{fig:flower-peanut-2d}(b, f) remains at the level of the nonlinear solver tolerance for both time step sizes, confirming the exact area-conservation property of the proposed scheme in these more challenging tests. The Newton iteration numbers in Fig.~\ref{fig:flower-peanut-2d}(c, g) stay moderate and decrease after the initial transient. Fig.~\ref{fig:flower-peanut-2d}(d, h) shows that the mesh ratio remains bounded and approaches one as the evolution proceeds; moreover, the smaller time step generally leads to a faster approach to the equidistributed state. These observations indicate that the proposed BGN/AC scheme remains robust for nonconvex or strongly nonuniform planar geometries.

\begin{figure}[htbp]
	\centering
	\includegraphics[width=1\linewidth]{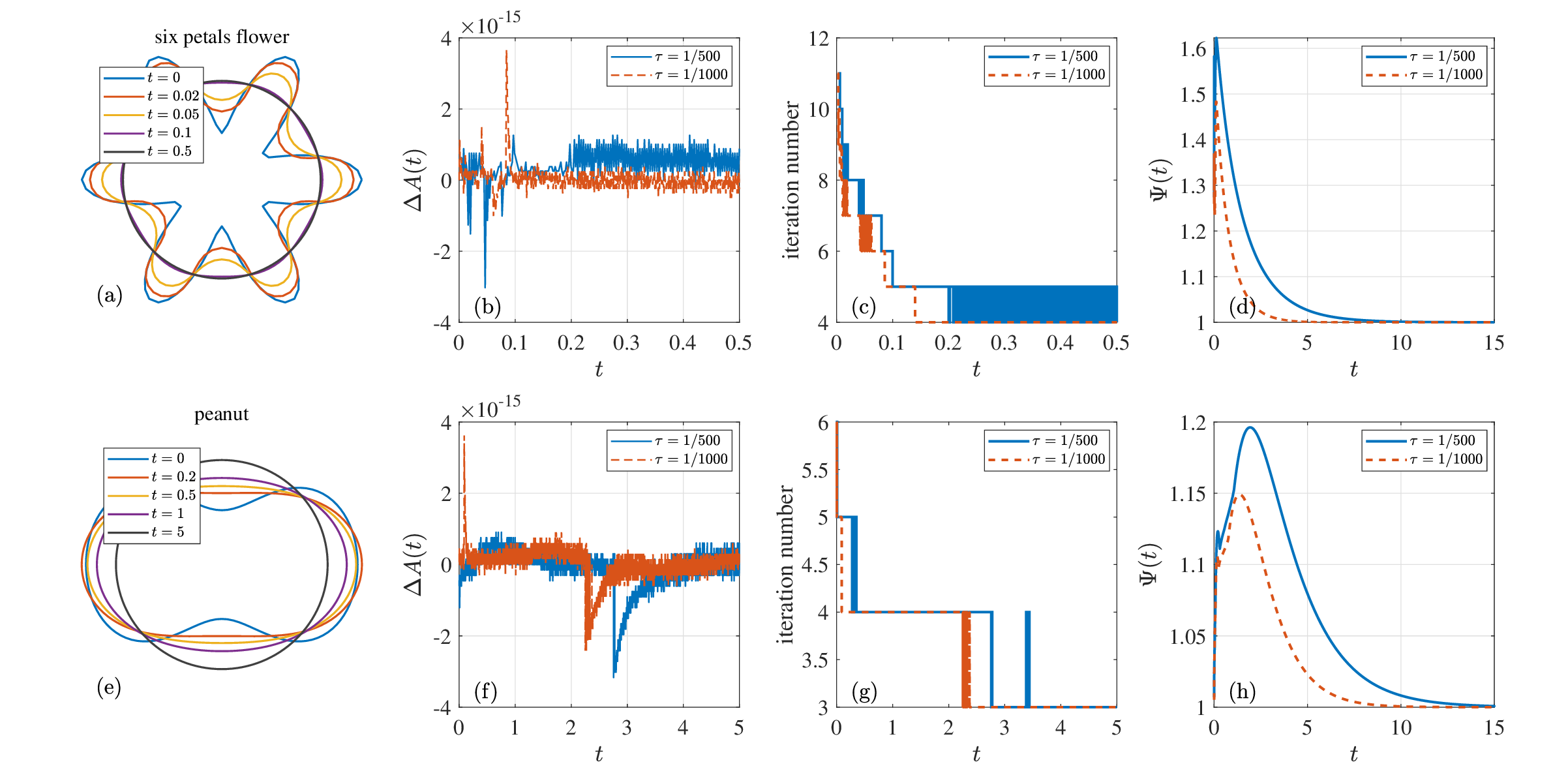}
	\caption{Geometric diagnostics produced by the proposed BGN/AC scheme for the six-petal flower curve (top row) and the peanut-shaped curve (bottom row) with $N=128$. For each initial curve, the diagnostic quantities in the last three columns are plotted for two time step sizes, $\tau=1/500$ and $\tau=1/1000$. (a, e) Snapshots of the evolution; (b, f) relative area loss $\Delta A(t)$; (c, g) Newton iteration number; (d, h) mesh ratio $\Psi(t)$.}
	\label{fig:flower-peanut-2d}
\end{figure}

\subsection{Surface experiments}

We now turn to the proposed BGN/VC scheme for closed surfaces in three dimensions. Two initial surfaces are considered. Throughout this subsection, $K$ and $J$ denote the numbers of vertices and triangular faces of the polyhedral surface, respectively.

\paragraph{Ellipsoid.}
\[
\frac{x^2}{4}+\frac{y^2}{(1.5)^2}+z^2=1.
\]
This $(2,1.5,1)$ ellipsoid is used for the convergence study and for verifying volume conservation, mesh behavior, and nonlinear solver performance.

\paragraph{Cigar-like surface.}
The cigar-like surface consists of a cylindrical part of radius $1/2$ and length $7$, capped by two hemispheres of radius $1/2$, so that the total length is $8$. This example is used to test the behavior of the method near a pinch-off event.
\\

We first study convergence for the $(2,1.5,1)$ ellipsoid. The initial polyhedral surfaces are generated by the distance-function mesh generator \texttt{DistMesh} \cite{persson2005mesh} with a uniform size function. The input parameter $h$ controls the spacing of the background Cartesian grid used in the isosurface extraction. Since the resulting meshes are nearly isotropic, the actual element diameters are proportional to $h$, and we use this parameter to label the refinement path. The four values $h=1/5,1/10,1/15,1/20$ yield triangulations with $(K,J)=(1028,2052)$, $(3990,7976)$, $(9196,18388)$, and $(16230,32456)$, respectively. The time step is chosen as $\tau=0.01h$.

To quantify the difference between two polyhedral surfaces $S_1$ and $S_2$, we use a symmetric vertex-to-surface distance, as in \cite{bao2021structure}, defined by
\[
M(S_1,S_2) := \frac{1}{2}\left(
\max_{\mathbf{q}_1 \in \mathcal{N}_1} \min_{\mathbf{p}_2 \in S_2} |\mathbf{q}_1 - \mathbf{p}_2|
+
\max_{\mathbf{q}_2 \in \mathcal{N}_2} \min_{\mathbf{p}_1 \in S_1} |\mathbf{q}_2 - \mathbf{p}_1|
\right),
\]
where $\mathcal{N}_i$ denotes the set of vertices of $S_i$. As in the curve case, the convergence error is computed by comparing numerical solutions on two successive refinements along the path $\tau=0.01h$. For two successive numerical solutions, we define
\[
\mathcal{E}^{3D}_{M}(T,\tau_1, \tau_2)
:=
M\!\left(S^{T/{\tau_1}}_{h_1,\tau_1},\, S^{T/{\tau_2}}_{h_2,\tau_{2}}\right).
\]
The errors are measured at $T=0.1,0.2,0.3$. The results in Table~\ref{tab:3d-convergence} show rates close to two, which supports the expected second-order accuracy of the fully discrete BGN/VC scheme.
\begin{table}[H]
	\centering
	\renewcommand{\arraystretch}{1.2}
	\caption{Cauchy-type convergence test for the proposed BGN/VC scheme applied to the surface diffusion of the $(2,1.5,1)$ ellipsoid. Here $\tau_0=1/500$ corresponds to the coarsest mesh with $h_0=1/5$.}
	\label{tab:3d-convergence}
	\begin{tabular}{c|cc|cc|cc}
		\hline
		$\tau$
		& $\mathcal{E}^{3D}_{M}(0.1)$ & order
		& $\mathcal{E}^{3D}_{M}(0.2)$ & order
		& $\mathcal{E}^{3D}_{M}(0.3)$ & order\\
		\hline
		$\tau_0$     & $4.35\times10^{-3}$ & -- 
		& $3.86\times10^{-3}$ & -- 
		& $3.54\times10^{-3}$ & -- \\
		$\tau_0/2$   & $1.26\times10^{-3}$ & $1.78$
		& $1.12\times10^{-3}$ & $1.79$
		& $1.04\times10^{-3}$ & $1.76$ \\
		$\tau_0/3$   & $5.94\times10^{-4}$ & $1.86$
		& $5.44\times10^{-4}$ & $1.78$
		& $4.99\times10^{-4}$ & $1.82$ \\
		\hline
	\end{tabular}
\end{table}

We next examine the geometric behavior of the proposed method for the same ellipsoidal surface. The computational parameters are
$
(K,J)=(3990,7976), \tau=0.001.$
The results are shown in Fig.~\ref{fig:ellipsoid-3d-diagnostics}. The relative volume loss remains at the level of the nonlinear solver tolerance throughout the computation, in agreement with Theorem~\ref{thm:volume-conservation-3d}; see Fig.~\ref{fig:ellipsoid-3d-diagnostics}(a). Fig.~\ref{fig:ellipsoid-3d-diagnostics}(b) shows that the triangle-area ratio remains well controlled during the evolution, indicating that the mesh quality does not deteriorate. The Newton iteration number remains low and stable in Fig.~\ref{fig:ellipsoid-3d-diagnostics}(c), which suggests that the nonlinear systems can be solved efficiently.

\begin{figure}[htbp]
	\centering
	\includegraphics[width=1\linewidth]{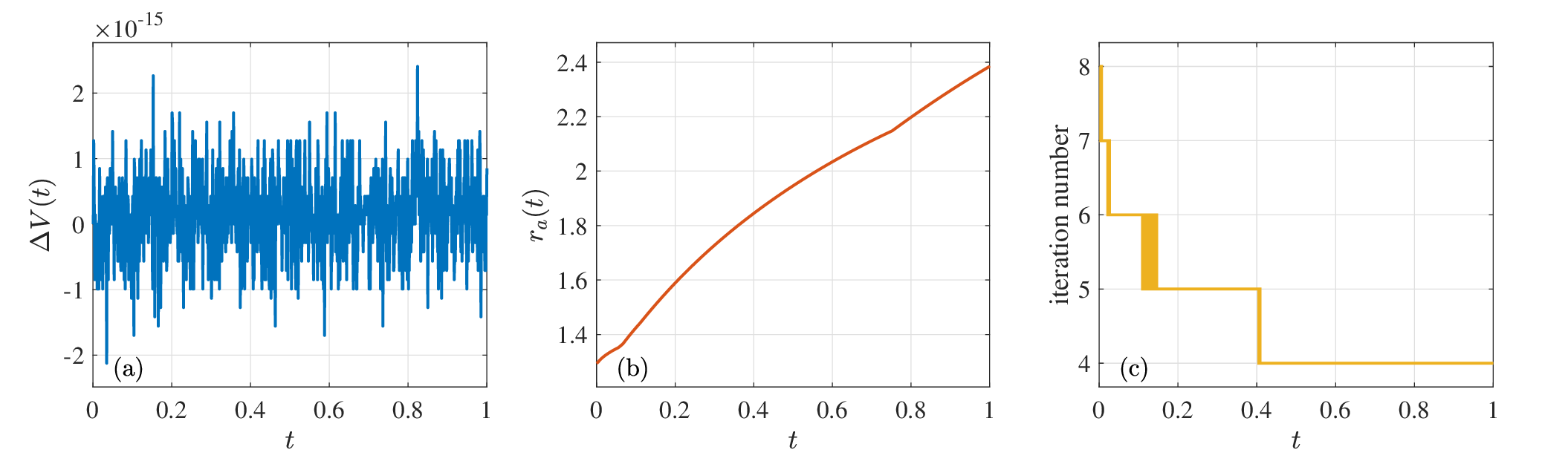}
	\caption{Geometric diagnostics for the proposed BGN/VC scheme applied to the surface diffusion of the $(2,1.5,1)$ ellipsoid with $(K,J)=(3990,7976)$ and $\tau=0.001$: (a) relative volume loss $\Delta V(t)$; (b) triangle area ratio $r_a$; (c) Newton iteration number.}
	\label{fig:ellipsoid-3d-diagnostics}
\end{figure}

We finally consider the cigar-like surface. This example is more challenging because the evolution develops a necking instability and approaches pinch-off. We compare the proposed BGN/VC scheme with the BGN/BDF2 scheme and the BGN/PC scheme. The BGN/CNLF scheme is not included here, since its mesh deterioration for this three-dimensional near-pinch-off example has already been reported in the predictor--corrector BGN study. The discretization parameters are
$
(K,J)=(1322,2640), \tau=0.001.
$
Representative snapshots of the evolving surfaces are shown in Fig.~\ref{fig:cigar-snapshots}. All three methods capture the same qualitative evolution: the initially elongated surface forms a thin neck, the neck radius decreases rapidly, and the computation terminates near pinch-off. The resulting pinch-off times are close, although not identical, for the three methods.

\begin{figure}[htbp]
	\centering
	\includegraphics[width=1\linewidth]{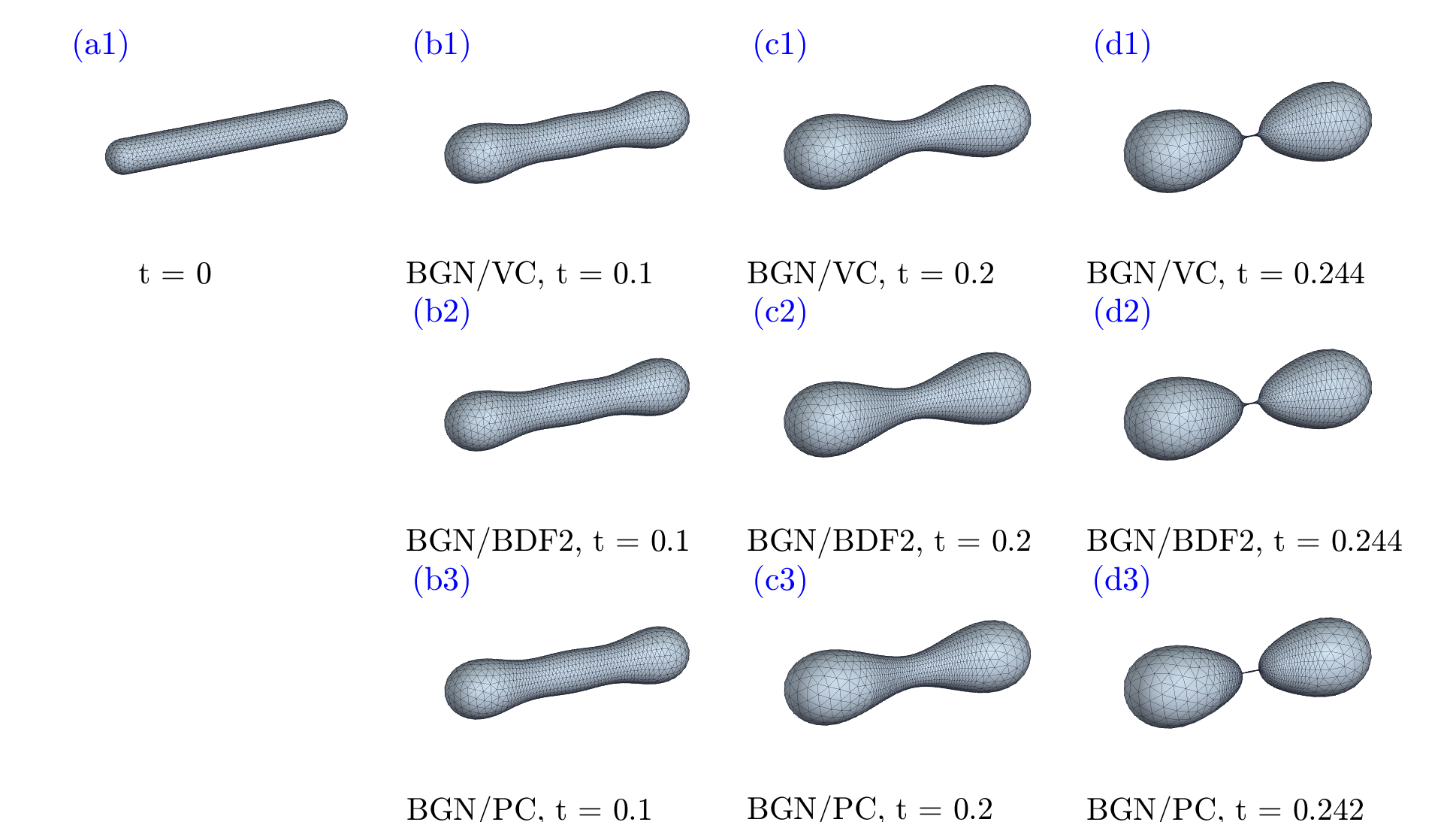}
	\caption{Comparison of surface evolution for the cigar-like initial surface with $(K,J)=(1322,2640)$ and $\tau=0.001$. The three rows correspond to the proposed BGN/VC scheme, the BGN/BDF2 scheme, and the BGN/PC scheme, respectively. The first three columns show representative snapshots during the evolution, while the last column shows the last computed surface before pinch-off for each method.}
	\label{fig:cigar-snapshots}
\end{figure}

The corresponding geometric diagnostics are reported in Fig.~\ref{fig:cigar-diagnostics}. Fig.~\ref{fig:cigar-diagnostics}(a) shows that all three methods produce a monotone decay of the normalized surface area. The zoom-in view near the pinch-off regime further indicates that the proposed BGN/VC scheme follows the BGN/BDF2 scheme very closely, while the BGN/PC scheme shows a slightly different decay profile near the pinch-off regime. The same tendency is more visible in the mesh-quality indicator shown in Fig.~\ref{fig:cigar-diagnostics}(c): the triangle-area ratio generated by BGN/VC remains close to that of BGN/BDF2 and is noticeably smaller than that of BGN/PC as the surface approaches pinch-off. 

The main difference among the three schemes lies in the volume behavior. As shown in Fig.~\ref{fig:cigar-diagnostics}(b), the proposed BGN/VC scheme keeps the absolute relative volume loss at the level of the nonlinear solver tolerance, whereas the BGN/BDF2 and BGN/PC schemes exhibit visible volume drift of order $10^{-5}$. These results indicate that the proposed volume-preserving modification retains the near-pinch-off dynamics and mesh behavior of the BGN/BDF2 scheme, while additionally enforcing exact volume conservation.

\begin{figure}[htbp]
	\centering
	\includegraphics[width=1\linewidth]{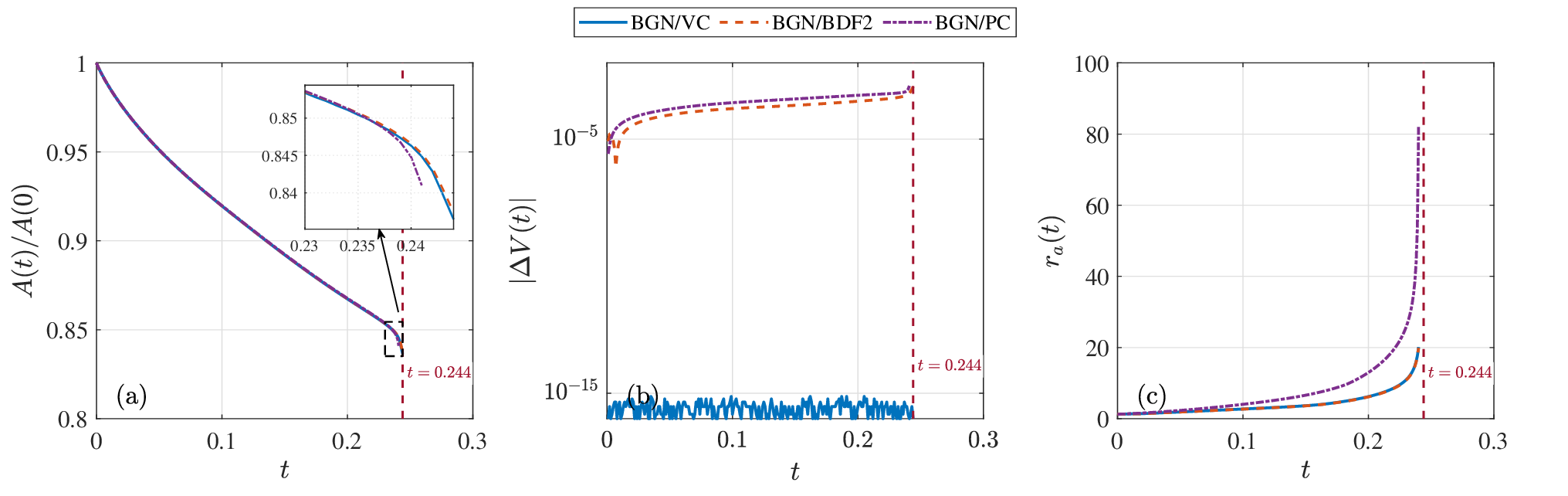}
	\caption{Comparison of geometric quantities for the cigar-like surface with $(K,J)=(1322,2640)$ and $\tau=0.001$: (a) normalized surface area $A(t)/A(0)$, together with a zoom-in view near the pinch-off regime; (b) absolute relative volume loss $|\Delta V(t)|$; (c) triangle-area ratio $r_a(t)$. The three curves correspond to the proposed BGN/VC scheme, the BGN/BDF2 scheme, and the BGN/PC scheme, respectively.}
	\label{fig:cigar-diagnostics}
\end{figure}
	
	\section{Conclusions}
	
	We have proposed a geometric-interpolation approach for constructing second-order area- and volume-preserving PFEMs for surface diffusion. The main point of the construction is that the conservative normal velocity is not imposed by an auxiliary constraint equation, but is obtained from the exact variation of the enclosed area or volume along a quadratic temporal interpolation path. This leads to a Simpson-type identity for closed curves and a Boole-type identity for closed surfaces.
	
	The resulting BGN/AC and BGN/VC schemes preserve the enclosed area or volume up to the nonlinear solver tolerance. At the same time, they can be implemented on BGN-predicted auxiliary geometries and therefore remain close to existing second-order BGN-based discretizations in terms of mesh behavior and practical performance. The numerical experiments confirm the expected second-order convergence, exact conservation, stable Newton iteration, and good mesh quality for both planar curves and three-dimensional surfaces, including near-pinch-off surface diffusion dynamics.
	
	The present work focuses on the conservation mechanism. As discussed above, a Lagrange-multiplier energy equation can be incorporated if one also wants to enforce a discrete perimeter or surface-area dissipation law. Other natural extensions include higher-order geometric interpolation formulas, anisotropic surface diffusion, and more general geometric flows with multiple constraints.

	\appendix

	\section{Finite difference coefficients on equidistant nodes}
	\label{app:fd-coeff}
	This appendix records a general formula for the finite-difference coefficients obtained by differentiating the polynomial interpolant on equidistant temporal nodes. The derivative approximations used in \eqref{eq:D-mminus}--\eqref{eq:D-mplus} are precisely the quadratic ($k=2$) special cases of this result. We include this statement both to place the three-level formulas used in the main text into a unified interpolation framework and to indicate how similar coefficients arise in possible higher-order extensions.
	\begin{theorem}[Finite difference coefficients]
		Let $x_j=j/k$, $j=0,\ldots,k$, and let $l_j(x)$ be the associated Lagrange basis. For each $n=0,\ldots,k$, the unique coefficients $b_{n,j}$ satisfying
		\[
		f'(x_n)=\sum_{j=0}^{k} b_{n,j}f(x_j),
		\qquad \forall f\in\mathbb{P}_k,
		\]
		are $b_{n,j}=l_j'(x_n)$ and are explicitly given by
		\[
		b_{n,j}=
		\begin{cases}
			\displaystyle k\left(H_n-H_{k-n}\right),& j=n,\\[1ex]
			\displaystyle k\,\frac{(-1)^{n-j}}{n-j}\frac{\binom{k}{j}}{\binom{k}{n}},& j\ne n,
		\end{cases}
		\]
		where $H_p=\sum_{i=1}^p i^{-1}$ and $H_0=0$. If $x=(t-t_0)/(k\tau)$, then
		\[
		\partial_t f(t_n)
		=
		\frac{1}{k\tau}
		\sum_{j=0}^{k}b_{n,j}f(t_j).
		\]
	\end{theorem}
	
	\begin{proof}
		Let $l_j(x)$ be the Lagrange basis associated with the nodes $x_i=i/k$. For any polynomial $f$ of degree at most $k$,
		\[
		f(x)=\sum_{j=0}^{k} f(x_j)l_j(x).
		\]
		Differentiating and evaluating at $x=x_n$ gives
		\[
		f'(x_n)=\sum_{j=0}^{k}f(x_j)l_j'(x_n),
		\]
		so $b_{n,j}=l_j'(x_n)$.
		
		Let $u=kx$ and define $\ell_j(u):=l_j(u/k)$. Then $l_j(x)=\ell_j(kx)$ and hence
		\[
		l_j'(x)=k\ell_j'(kx).
		\]
		Thus
		\[
		l_j'(x_n)=k\ell_j'(n).
		\]
		For integer nodes $u_i=i$,
		\[
		\ell_j(u)=\prod_{\substack{m=0\\m\ne j}}^{k}\frac{u-m}{j-m}.
		\]
		If $j=n$, then
		\[
		\ell_n'(n)=
		\sum_{\substack{m=0\\m\ne n}}^{k}\frac{1}{n-m}
		=
		H_n-H_{k-n}.
		\]
		If $j\ne n$, then
		\[
		\ell_j'(n)
		=
		\frac{1}{n-j}
		\prod_{\substack{m=0\\m\ne n,j}}^{k}
		\frac{n-m}{j-m}
		=
		\frac{(-1)^{n-j}}{n-j}
		\frac{\binom{k}{j}}{\binom{k}{n}}.
		\]
		Multiplying by $k$ gives the stated formula.
	\end{proof}
	
	\section{Iterative solver for the 3D nonlinear system}
	\label{app:solver-3d}
	
	In this appendix, we describe the Newton-type solver used for the three-dimensional volume-preserving scheme \eqref{eq:scheme-3d}. Compared with the two-dimensional case, the main additional ingredient is the linearization of the oriented area vector. On a reference triangle with vertices mapped to $\mathbf z_1,\mathbf z_2,\mathbf z_3$, this vector is
	\[
	\mathcal J(\mathbf X)
	=
	(\mathbf z_2-\mathbf z_1)\times(\mathbf z_3-\mathbf z_1).
	\]
	For a perturbation $\boldsymbol\eta$, its first variation is
	\begin{equation}
		D\mathcal J(\mathbf X)[\boldsymbol\eta]
		=
		(\boldsymbol\eta_2-\boldsymbol\eta_1)\times(\mathbf z_3-\mathbf z_1)
		+
		(\mathbf z_2-\mathbf z_1)\times(\boldsymbol\eta_3-\boldsymbol\eta_1).
		\label{eq:DJ-3d}
	\end{equation}
	
	Let
	\[
	(\mathbf{X}^{m+1,l},H^{m+1,l})
	\]
	be the approximation at the $l$-th Newton iteration, and define the increments
	\[
	\delta\mathbf{X}^{l}
	:=
	\mathbf{X}^{m+1,l+1}-\mathbf{X}^{m+1,l},
	\qquad
	\delta H^{l}
	:=
	H^{m+1,l+1}-H^{m+1,l}.
	\]
	
	The two half-time geometries induced by the quadratic interpolation path are
	\begin{align}
		\mathbf{X}^{m+\frac12,l}
		&:=
		-\frac18\mathbf{X}^{m-1}
		+\frac34\mathbf{X}^{m}
		+\frac38\mathbf{X}^{m+1,l},
		\\
		\mathbf{X}^{m-\frac12,l}
		&:=
		\frac38\mathbf{X}^{m-1}
		+\frac34\mathbf{X}^{m}
		-\frac18\mathbf{X}^{m+1,l}.
	\end{align}
	Hence their variations satisfy
	\begin{equation}
		\delta\mathbf{X}^{m+\frac12,l}
		=
		\frac38\delta\mathbf{X}^{l},
		\qquad
		\delta\mathbf{X}^{m-\frac12,l}
		=
		-\frac18\delta\mathbf{X}^{l}.
		\label{eq:delta-half-Y}
	\end{equation}
	
	At the Newton iterate, the structure-preserving weighted normals are defined by
	\begin{align}
		\mathbf{n}^{m+1,l}
		&:=
		\frac{\mathcal{J}(\mathbf{X}^{m+1,l})}
		{|\mathcal{J}(\widetilde{\mathbf{X}}^{m+1})|},
		\\
		\mathbf{n}^{m+\frac12,l}
		&:=
		\frac{\mathcal{J}(\mathbf{X}^{m+\frac12,l})}
		{|\mathcal{J}(\widetilde{\mathbf{X}}^{m+\frac12})|},
		\\
		\mathbf{n}^{m-\frac12,l}
		&:=
		\frac{\mathcal{J}(\mathbf{X}^{m-\frac12,l})}
		{|\mathcal{J}(\widetilde{\mathbf{X}}^{m-\frac12})|}.
	\end{align}
	Using \eqref{eq:DJ-3d} and \eqref{eq:delta-half-Y}, their first variations are
	\begin{align}
		\delta\mathbf{n}^{m+1,l}
		&=
		\frac{
			D\mathcal{J}(\mathbf{X}^{m+1,l})[\delta\mathbf{X}^{l}]
		}{
			|\mathcal{J}(\widetilde{\mathbf{X}}^{m+1})|
		},
		\label{eq:dn-mplus-3d}
		\\
		\delta\mathbf{n}^{m+\frac12,l}
		&=
		\frac{
			D\mathcal{J}(\mathbf{X}^{m+\frac12,l})[\frac38\delta\mathbf{X}^{l}]
		}{
			|\mathcal{J}(\widetilde{\mathbf{X}}^{m+\frac12})|
		},
		\label{eq:dn-phalf-3d}
		\\
		\delta\mathbf{n}^{m-\frac12,l}
		&=
		\frac{
			D\mathcal{J}(\mathbf{X}^{m-\frac12,l})[-\frac18\delta\mathbf{X}^{l}]
		}{
			|\mathcal{J}(\widetilde{\mathbf{X}}^{m-\frac12})|
		}.
		\label{eq:dn-mhalf-3d}
	\end{align}
	
	For convenience, we introduce the velocity approximations appearing in the five Boole terms:
	\begin{align}
		\mathbf{D}^{m+1,l}
		&:=
		\frac{3\mathbf{X}^{m+1,l}-4\mathbf{X}^{m}+\mathbf{X}^{m-1}}{2\tau},
		&
		\delta\mathbf{D}^{m+1,l}
		&=
		\frac{3}{2\tau}\delta\mathbf{X}^{l},
		\\
		\mathbf{D}^{m+\frac12,l}
		&:=
		\frac{\mathbf{X}^{m+1,l}-\mathbf{X}^{m}}{\tau},
		&
		\delta\mathbf{D}^{m+\frac12,l}
		&=
		\frac{1}{\tau}\delta\mathbf{X}^{l},
		\\
		\mathbf{D}^{m,l}
		&:=
		\frac{\mathbf{X}^{m+1,l}-\mathbf{X}^{m-1}}{2\tau},
		&
		\delta\mathbf{D}^{m,l}
		&=
		\frac{1}{2\tau}\delta\mathbf{X}^{l},
		\\
		\mathbf{D}^{m-\frac12}
		&:=
		\frac{\mathbf{X}^{m}-\mathbf{X}^{m-1}}{\tau},
		\\
		\mathbf{D}^{m-1,l}
		&:=
		\frac{-\mathbf{X}^{m+1,l}+4\mathbf{X}^{m}-3\mathbf{X}^{m-1}}{2\tau},
		&
		\delta\mathbf{D}^{m-1,l}
		&=
		-\frac{1}{2\tau}\delta\mathbf{X}^{l}.
	\end{align}
	
	At each Newton step, we solve for
	\[
	(\delta\mathbf{X}^{l},\delta H^{l})\in[\mathcal V^h]^3\times\mathcal V^h
	\]
	such that, for all $(\varphi^h,\boldsymbol{\omega}^h)\in\mathcal V^h\times[\mathcal V^h]^3$,
	\begin{subequations}
		\label{eq:newton-3d}
		\begin{align}
			&
			\frac{7}{90}
			\Bigg[
			\left(
			\delta\mathbf{D}^{m+1,l},
			\varphi^h\mathbf{n}^{m+1,l}
			\right)^h_{\widetilde{S}^{m+1}}
			+
			\left(
			\mathbf{D}^{m+1,l},
			\varphi^h\delta\mathbf{n}^{m+1,l}
			\right)^h_{\widetilde{S}^{m+1}}
			+
			\left(
			\nabla_{\widetilde{S}^{m+1}}\delta H^{l},
			\nabla_{\widetilde{S}^{m+1}}\varphi^h
			\right)_{\widetilde{S}^{m+1}}
			\Bigg]
			\nonumber\\
			&\quad
			+
			\frac{32}{90}
			\Bigg[
			\left(
			\delta\mathbf{D}^{m+\frac12,l},
			\varphi^h\mathbf{n}^{m+\frac12,l}
			\right)^h_{\widetilde{S}^{m+\frac12}}
			+
			\left(
			\mathbf{D}^{m+\frac12,l},
			\varphi^h\delta\mathbf{n}^{m+\frac12,l}
			\right)^h_{\widetilde{S}^{m+\frac12}}
			+
			\left(
			\nabla_{\widetilde{S}^{m+\frac12}}\frac{\delta H^{l}}{2},
			\nabla_{\widetilde{S}^{m+\frac12}}\varphi^h
			\right)_{\widetilde{S}^{m+\frac12}}
			\Bigg]
			\nonumber\\
			&\quad
			+
			\frac{12}{90}
			\left(
			\delta\mathbf{D}^{m,l},
			\varphi^h\mathbf{n}^{m}
			\right)^h_{S^{m}}
			+
			\frac{32}{90}
			\left(
			\mathbf{D}^{m-\frac12},
			\varphi^h\delta\mathbf{n}^{m-\frac12,l}
			\right)^h_{\widetilde{S}^{m-\frac12}}
			+
			\frac{7}{90}
			\left(
			\delta\mathbf{D}^{m-1,l},
			\varphi^h\mathbf{n}^{m-1}
			\right)^h_{S^{m-1}}
			\nonumber\\
			&\qquad
			=
			-\mathcal{R}_1^{l}(\varphi^h),
			\label{eq:newton-3d-a}
			\\
			&
			\left(
			\delta H^{l},
			\widetilde{\mathbf{n}}^{m+1}\cdot\boldsymbol{\omega}^h
			\right)^h_{\widetilde{S}^{m+1}}
			-
			\left(
			\nabla_{\widetilde{S}^{m+1}}\delta\mathbf{X}^{l},
			\nabla_{\widetilde{S}^{m+1}}\boldsymbol{\omega}^h
			\right)_{\widetilde{S}^{m+1}}
			=
			-\mathcal{R}_2^{l}(\boldsymbol{\omega}^h),
			\label{eq:newton-3d-b}
		\end{align}
	\end{subequations}
	where $\mathcal{R}_1^l$ and $\mathcal{R}_2^l$ denote the residuals of \eqref{eq:scheme-3d-a} and \eqref{eq:scheme-3d-b}, evaluated at $(\mathbf{X}^{m+1,l},H^{m+1,l})$.
	
	After solving \eqref{eq:newton-3d}, we update
	\[
	\mathbf{X}^{m+1,l+1}
	=
	\mathbf{X}^{m+1,l}+\delta\mathbf{X}^{l},
	\qquad
	H^{m+1,l+1}
	=
	H^{m+1,l}+\delta H^{l}.
	\]
	The iteration is terminated once
	\[
	\max\left\{
	\|\delta\mathbf{X}^{l}\|_{\ell^\infty},
	\|\delta H^{l}\|_{\ell^\infty}
	\right\}
	\le \mathrm{tol}.
	\]
	
	In the computations, the initial guess is taken as
	\[
	\mathbf{X}^{m+1,0}=\widetilde{\mathbf{X}}^{m+1},
	\qquad
	H^{m+1,0}=H^m,
	\]
	where the predicted surfaces $\widetilde{S}^{m+1}$ and $\widetilde{S}^{m+\frac12}$ are obtained by BGN1 predictor steps with time steps $\tau$ and $\tau/2$, respectively.
	
	\begin{remark}
		The exact volume conservation of the fully discrete scheme is recovered up to the nonlinear solver tolerance. The principal additional feature in the three-dimensional Newton linearization is the consistent treatment of the weighted normals at the new and half-time geometries through the first variation of the oriented area vector \eqref{eq:DJ-3d}.
	\end{remark}

	\section*{Acknowledgment}
	This work was partially supported by Shenzhen Loop Area Institute (Contract No. SLAI2026020007).

	\section*{Data Availability}
	The datasets generated during and/or analyzed during the current study are available from
	the authors on reasonable request.
	\section*{Declarations}
	\section*{Conflict of interest}
	The authors declare that they have no conflict of interest.
	\bibliographystyle{spmpsci}
	\bibliography{references}
	
	\end{document}